\documentclass[12pt]{amsart}

\usepackage{amssymb}

\theoremstyle{definition}
\newtheorem{thm}{Theorem}[section]
\newtheorem{lem}[thm]{Lemma}
\newtheorem{defn}[thm]{Definition}
\newtheorem{cor}[thm]{Corollary}
\newtheorem{prop}[thm]{Proposition}
\newtheorem{ex}[thm]{Example}
\theoremstyle{remark}
\newtheorem{rem}[thm]{Remark}

\newtheorem{obs}[thm]{Observation}

\newcommand{\N}{{\mathbb{N}}}
\newcommand{\R}{{\mathbb{R}}}
\newcommand{\Z}{{\mathbb{Z}}}
\newcommand{\T}{{\mathbb{T}}}
\newcommand{\C}{{\mathbb{C}}}

\newcommand{\cI}{{\mathcal I}}
\newcommand{\cO}{{\mathcal O}}
\newcommand{\cP}{{\mathcal P}}

\newcommand{\sgn}{\mathop{\rm sgn}}

\newcommand{\kernel}{\mathop{\rm ker}}

\newcommand{\spec}{\mathop{\rm Spec}}

\usepackage{url}

\def\primitive{\frak{p}}
\def\phi{\varphi}

\def\deg{\mathrm{deg}}

\def\t{\mathbf{t}}
\def\z{\mathbf{z}}
\def\w{\mathbf{w}}
\def\v{\mathbf{v}}
\def\zero{\mathbf{0}}
\def\one{\mathbf{1}}
\def\bra{\langle}
\def\ket{\rangle}
\def\im{\mathrm{Im}}

\def\lup{L^{\mathrm{up}}}
\def\ldown{L^{\mathrm{down}}}
\def\aup{A^{\mathrm{up}}}
\def\adown{A^{\mathrm{down}}}
\def\lup{L^{\mathrm{up}}}
\def\ldown{L^{\mathrm{down}}}
\def\Vhat{\hat{V}}
\def\Psiprime{\Psi} 
\def\vec#1{\mathbf{#1}}
\def\aa{\vec{a}}

\def\mm{m,m}
\newcommand{\lcm}{\mathrm{lcm}}
\newcommand{\orb}{\mathrm{orb}}
\newcommand{\Gal}{\mathrm{Gal}}
\newcommand{\Psiirred}{\Psi^{\mathrm{irr}}}
\newcommand{\half}{(\mathrm{half})}

\numberwithin{equation}{section}

\def\la{\lambda}

\def\t{{\bf t}}

\def\Q{\mathbb{Q}}

\newcommand\lref[1]{Lemma~\ref{#1}}
\newcommand\tref[1]{Theorem~\ref{#1}}
\newcommand\pref[1]{Proposition~\ref{#1}}
\newcommand\sref[1]{Section~\ref{#1}}

\newcommand\rref[1]{Remark~\ref{#1}} 
\newcommand\cref[1]{Corollary~\ref{#1}}
\newcommand\eref[1]{Example~\ref{#1}}

\title{Zeta functions of periodic cubical lattices and cyclotomic-like polynomials}
\author{Yasuaki Hiraoka}
\address{Kyoto University Institute for Advanced Study, WPI
Institute for the Advanced Study of Human Biology, Kyoto
University, Yoshida Ushinomiya-cho, Sakyo-ku, Kyoto
606-8501, Japan; 
Center for Advanced Intelligence Project,
RIKEN, 1-4-1 Nihonbashi, Chuo-ku, Tokyo, 103-0027, Japan}
\email{hiraoka.yasuaki.6z@kyoto-u.ac.jp}
\author{Hiroyuki Ochiai}
\address{Institute of Mathematics for Industry, 
Kyushu University, 
744 Motooka, Nishi-ku,
Fukuoka 819-0395, Japan} 
\email{ochiai@imi.kyushu-u.ac.jp}
\author{Tomoyuki Shirai}
\address{Institute of Mathematics for Industry, 
Kyushu University, 
744 Motooka, Nishi-ku,
Fukuoka 819-0395, Japan; 
Center for Advanced Intelligence Project,
RIKEN, 1-4-1 Nihonbashi, Chuo-ku, Tokyo, 103-0027, Japan}
\email{shirai@imi.kyushu-u.ac.jp}

\subjclass[2010]{11R09, 11S40, 05E45, 05C50, 58C40}
\keywords{Cubical lattice, zeta function, Laplacian,
cyclotomic-like polynomial}
\begin{document}
\begin{abstract} Zeta functions of 
periodic cubical lattices are explicitly derived  
 by computing all the eigenvalues of the adjacency operators and their
 characteristic polynomials. 
We introduce cyclotomic-like polynomials to give 
factorization of the zeta function in terms of them and 
count the number of orbits of the Galois action 
associated with each cyclotomic-like polynomial to 
obtain its further factorization. 
We also give a necessary and sufficient condition 
for such a polynomial to be irreducible and discuss its 
 irreducibility from this point of view.  
\end{abstract}
\maketitle

\section{Introduction}\label{sec:intro}

Let $K$ be an abstract simplicial complex over the set
$K_0=\{1,\dots,q\}$ of vertices, where we assume that the empty set
$\emptyset$ is an element of $K$, i.e.,
$K_{-1}=\{\emptyset\}$. 
We denote 
the set of $k$-simplices and the $k$-dimensional skeleton of $K$
by $K_k$ and $K^{(k)}(=\sqcup_{i=-1}^k K_i)$. For each simplex
$\sigma\in K$, $|\sigma|$ denotes the number of vertices in $\sigma$,
and similarly, $|S|$ denotes the cardinality of a finite set $S$.  

Let $\vec{n}=(n_1,\dots,n_q)\in\N^q$ be a fixed vector of natural
numbers with $n_i \ge 2$ for all $i \in K_0$. 
We write $|\vec{n}| = \prod_{i=1}^q n_i$. 
For each $i\in K_0$, let 
\[
	\cI_i^{\circ}:=\{[0,1],[1,2], \dots,[n_i-1,0]\}, \quad
	\cP_i^{\circ}:=\{[0,0],[1,1], \dots,[n_i-1,n_i-1]\} 
\]
be the collections of intervals in $\R/n_i\Z$, where the
elements of $\cP^{\circ}_i$ 
are degenerated intervals. 
For each simplex $\sigma\in K$, we set a collection of $|\sigma|$-dimensional elementary cubes by 
\[
Y_\sigma=
\{
I_1\times\dots\times I_q \colon
I_i\in\cI^{\circ}_i~{\rm if~}  i\in\sigma;~
I_i\in\cP^{\circ}_i~{\rm otherwise}
\}.
\]
Then, we define the $d$-cubical lattice $Y^{(d)}$ over $K$
by 
\[
	Y^{(d)}=\bigsqcup _{k=0}^d Y_k,\quad Y_k=\bigsqcup_{\sigma\in K_{k-1}} Y_\sigma
\]
for $d\leq q$. We also write $Y=Y^{(q)}$. 
When we need to specify the side lengths $\vec{n}$, we write
$Y = Y(\vec{n})$.  
We refer to an element in $Y_k$ as a $k$-cube. 
The cubical lattice $Y$ is realized in the $q$-dimensional torus $\T^q$. 
For more explanation about elementary cubes (and also cubical homology), 
we refer the reader to \cite{kmm}; in particular, 
see subsection 2.1.1. for elementary cubes. 

A closed path in $Y^{(d)}$ is an alternating sequence 
$c = (\tau_0,\sigma_0, \tau_1, \sigma_1, \dots, \tau_{n-1},\sigma_{n-1})$ of $(d-1)$-cubes
$(\tau_i)_{i=0}^{n-1}$ and $d$-cubes $(\sigma_i)_{i=0}^{n-1}$ such that
$\tau_i \not= \tau_{i+1}$ 
and $\tau_{i+1} \subset \sigma_i \cap \sigma_{i+1}$ 
for all $i \in \Z_n := \Z/n \Z$. Note that $\tau_0 \subset \sigma_{n-1} \cap \sigma_0$. 
The length of a closed path $c$, denoted by $|c|$, 
is the number of $d$-cubes (also $(d-1)$-cubes) in $c$. 
We say that a closed path $c$ is a closed geodesic if there is no
back-tracking, i.e., $\sigma_i \not= \sigma_{i+1}$ for all $i \in \Z_{|c|}$.  
We denote by $c^m$ the $m$-multiple of a closed geodesic
$c$, which is formed by $m$ repetitions of $c$. 
If a closed geodesic $c$ is not expressed as 
an $m$-multiple of a closed geodesic with $m \ge 2$, then $c$ is said to be prime. 
Two prime closed geodesics are said to be equivalent if one
is obtained from the other through a cyclic permutation. 
An equivalence class of prime closed geodesics is called a prime cycle. 
The length of a prime cycle $\primitive$ is the length of a representative and is denoted by $|\primitive|$. 
The (Ihara) zeta function of a finite cubical lattice is defined as follows. 
\begin{defn}\label{def:zetafn}
 For $u \in \C$ with $|u|$ sufficient small, the (Ihara)
 zeta function of the $d$-skeleton $Y^{(d)}$ of a cubical lattice $Y$ is defined by 
\[
 \zeta_{Y^{(d)}}(u) = \prod_{\primitive \in P_d} (1 - u^{|\primitive|})^{-1}, 
\]
where $P_d$ is the set of prime cycles of $Y^{(d)}$. 
In particular, we write $\zeta_Y(u)$ for $\zeta_{Y^{(q)}}(u)$.
\end{defn}
Although zeta functions can be defined in this manner for
any cubical (simplicial) complexes, in the present paper, 
we only consider zeta functions for the cubical 
lattice $Y$ with the complete complex $K$, which includes
all the subsets of $K_0$.

By symmetry, we have $\zeta_{Y^{(d)}}(u) = \zeta_{Y^{(q-d+1)}}(u)$ for $d=1,2,\dots,q$ (see \rref{rem:symmetry}). 
For $d=q, q-1$ (and hence, by symmetry, for $d=1,2$, respectively), 
we have explicit factorizations of the zeta functions
defined above. 
\begin{thm}\label{thm:zeta}
Let $Y$ be the periodic $q$-cubical lattice with 
side lengths $\vec{n} = (n_1,n_2,\dots, n_q)$. Then, for
 $d=q$, 
\begin{align*}
\zeta_{Y}(u)^{-1}
 &= (1-u^2)^{(q-1)|\vec{n}|}
 \prod_{k_1=1}^{n_1} \!\! \cdots \!\! \prod_{k_q=1}^{n_q}\!
\Big\{1 - 2u \sum_{i=1}^q \cos \frac{2\pi k_i}{n_i} + (2q-1)u^2 \Big\}  
\end{align*}
and for $d=q-1$, 
 \begin{align*}
 \zeta_{Y^{(q-1)}}(u)^{-1} 
 &= (1-u)^{\kappa |\vec{n}| } (1+3u)^{\gamma|\vec{n}|}
  \times
 \prod_{k_1=1}^{n_1} \cdots \prod_{k_q=1}^{n_q}
  F_1^{\mathrm{up}}(u, \vec{k}). 
 \end{align*}
 Here $\vec{k} = (k_1,\dots,k_q)$,
 $\kappa = q(3q-5)/2$, $\gamma = q(q-3)/2$ and 
\[
F_1^{\mathrm{up}}(u, \vec{k}) = \sum_{\ell=0}^q 
(2-\ell)2^{\ell-1} e_{\ell}(\w) 
 \Big\{ 1 - 2 u \sum_{i=1}^q \cos \frac{2 \pi k_i}{n_i} + 3(2q-3) u^2 \Big\}^{q-\ell} u^{\ell}, 
\]
where $e_{\ell}(\vec{t})$ is the $k$th elementary symmetric polynomial in 
$\t=(t_1,\dots,t_q)$, and $\w = (2 + 2 \cos \frac{2 \pi k_i}{n_i})_{i=1}^q$. 
\end{thm}

 In principle, the expression of $\zeta^{-1}_{Y^{(d)}}$
 for arbitrary $d$ can be obtained using \tref{thm:zetaforcubical}.

Let $Y = Y(\vec{n})$ be the periodic cubical lattice
 with side lengths $\vec{n} =(n_1,\dots,n_q)$. Then, for $|u| < 1$, 
 \begin{align*}
 \lefteqn{
 \lim_{n_1, \dots,n_q \to \infty} \frac{-1}{|\vec{n}|} \log
 \zeta_Y(u)}\\  
&= (q-1) \log (1-u^2)
+ \int_{[0,1]^q} \log \Big(1 - 2u \sum_{i=1}^q \cos 2
 \pi \theta_i + (2q-1)u^2 \Big) d\theta_1 \cdots d\theta_q. 
 \end{align*}
Note that the second term on the right-hand side is the
logarithmic Mahler measure of the Laurant polynomial $1 +
(2q-1)u^2 - u \sum_{i=1}^q (z_i + z_i^{-1}) \in
\C[z_1,z_1^{-1}, \dots, z_q, z_q^{-1}]$ (cf. \cite{boyd}). 
It is known that as $u \to 1$, 
the second term converges to $0$ and $-4G/\pi$ for $q=1$ and
$2$, respectively, where $G$ is the Catalan number defined by 
$G = \sum_{k=0}^{\infty}(-1)^k/(2k+1)^2$.

The zeta function given in \tref{thm:zeta} can be written as 
a product of cyclotomic-like polynomials with integer coefficients. 
Let 
\begin{equation}
J_1=J_2=\{1\},
\quad J_d = \{j \in \N : j < d/2, \ \gcd(j,d)=1\} \text{ for
$d \ge 3$}, 
\label{eq:Jd}
\end{equation}
and then
\[
 |J_d| = \tilde{\phi}(d) := 
\begin{cases}
 \phi(d)/2 & \text{for $d \ge 3$}, \\
 1 & \text{for $d =1,2$},  
\end{cases}
\]
where $\phi$ is the Euler function, i.e., 
$\varphi(d)$ is the cardinality of 
$\{j \in \N : j \le d, \ \gcd(j,d)=1\}$. 
The set $J_d$ can be regarded as a set of representatives of $(\Z / d\Z)^{\times}/
\sim$, where $\sim$ is the equivalence relation with respect
to the involution $\iota : (\Z / d\Z)^{\times} \to (\Z /
d\Z)^{\times}$ defined by 
$\iota(k) = d-k \mod d$. 

For $\vec{d} = (d_1,\dots,d_q) \in \N^q$, we define the
following homogeneous polynomial of $x$ and $y$ of degree
$\prod_{i=1}^q \tilde{\phi}(d_i)$: 
\[
 \Psiprime_{\vec{d}}(x,y) := 
\prod_{j_1 \in J_{d_1}} \cdots \prod_{j_q \in J_{d_q}} \Big(x -
2y\sum_{i=1}^q \cos \frac{2\pi j_i}{d_i}\Big).  
\]
Then, it is seen that $\Psiprime_{\vec{d}}(x,y) \in
\Z[x,y]$ (\pref{proposition:injective}), and in particular,
for $q=1$, $\Psi_d(x,y)$ is irreducible for any $d \in \N$
(\lref{lem:Psid}) (cf. \cite{WZ}). 
Although $\Psiprime_{\vec{d}}(x,y)$ may be reducible for 
general $\vec{d}$ unless $q = 1$, if $\vec{d} = (d_1,\dots,d_q)$ are
relatively prime, then $\Psi_{\vec{d}}$ is irreducible (\cref{lem:inj2irr}). 

Using the polynomials $\Psi_{\vec{d}}$, 
we can factorize the zeta functions as follows. 

\begin{cor}\label{thm:zeta2}
Let $Y$ be the periodic $q$-cubical lattice with 
side lengths $\vec{n} = (n_1,n_2,\dots, n_q)$. Then, 
\[
 \zeta_{Y}(u)^{-1} 
 = (1-u^2)^{(q-1)|\vec{n}|} 
\prod_{\vec{d} | \vec{n}}
\Psiprime_{\vec{d}}(1+(2q-1)u^2, u)^{\epsilon(\vec{d})}, 
\]
where $\vec{d} | \vec{n}$ means $d_i | n_i$ for all $i=1,2,\dots,q$ and 
$\epsilon(\vec{d}) = 2^k$ with $k = \#\{1 \le i \le q : d_i \ge 3\}$. 
\end{cor}

For $\vec{d} = (d_1,\dots, d_q)$, 
let $V = V_{\vec{d}} = \{j \in \{1,2,\dots,q\} : d_j \ge 3\}$. 
We define a graph $\Gamma(V)$ obtained from $V$ by adding an unoriented edge
between each pair of distinct $i$ and $j$ satisfying $\gcd(d_i,d_j) \ge 3$.  
We denote by $\tilde{\beta}_0(\Gamma(V))$ the $0$th reduced Betti
number, i.e., one less than 
the number of connected components of $\Gamma(V)$. 
We understand $\tilde{\beta}_0(\Gamma(\emptyset))
= 0$ when $V = \emptyset$. Then we have the following. 
\begin{thm}\label{thm:numoforbits}
For $\vec{d} = (d_1,d_2,\dots,d_q)$, let $N=d_1\cdots d_q$
 and denote the number of 
$(\Z/N\Z)^{\times}$-orbits in $J_{d_1} \times \cdots \times
 J_{d_q}$ by $\orb(\vec{d})$, where 
the $(\Z/N\Z)^{\times}$ acts on $J_{d_1} \times \cdots
\times J_{d_q}$ as the component-wise mulitiplication, 
i.e., $a \cdot (j_1, \dots, j_q) := 
(a j_1, \dots, a j_q)$ for $a \in (\Z/N\Z)^{\times}$. 
Then,  
\[
 \orb(\vec{d})
 = \frac{\prod_{i=1}^q
 \tilde{\phi}(d_i)}{\tilde{\phi}(\lcm(d_1,d_2,\dots,d_q))}
 \times 2^{\tilde{\beta}_0(\Gamma(V_{\vec{d}}))}. 
\]
\end{thm}
While the above is the general form for the number of
orbits, in certain special cases, this is greatly
simplified as seen below. 
\begin{ex}\label{ex:numberoforbits}
\textup{(i)} For $q=2$, 
$\orb(d_1,d_2) = \tilde{\phi}(\gcd(d_1,d_2))$. \\ 
\textup{(ii)} For $d_1=d_2=\cdots = d_q = d$, 
 $\orb(\vec{d}) = \tilde{\phi}(d)^{q-1}$. 
\end{ex}

Next, for the purpose of factorizing $\Psi_{\vec{d}}(x)$, 
we introduce a new polynomial. 
For a subset $\mathcal{O} \subset J_{d_1} \times \cdots \times J_{d_q}$,
we define
\begin{equation}
\Psi_{\vec{d}}(x; \mathcal{O})
:=
\prod_{(j_1,\ldots,j_q) \in \mathcal{O}} \Big(x -
2\sum_{i=1}^q \cos \frac{2\pi j_i}{d_i}\Big).  
\label{eq:Psiorbit}
\end{equation}
It is obvious that this function has the following
multiplicativity for the disjoint union:
\[
\Psi_{\vec{d}}(x; \mathcal{O}_1 \sqcup \mathcal{O}_2)
= \Psi_{\vec{d}}(x; \mathcal{O}_1) \Psi_{\vec{d}}(x;
\mathcal{O}_2). 
\]
From \tref{thm:numoforbits}, we obtain the following
factorization. 
\begin{cor}\label{cor:further_factor}
Let $\sqcup_{i=1}^{\orb(\vec{d})} \mathcal{O}_i$ 
be the decomposition of $J_{d_1} \times\cdots\times J_{d_q}$ 
into $(\Z/N\Z)^\times$-orbits.
Then, 
\begin{equation}
\Psiprime_{\vec{d}}(x) = \prod_{i=1}^{\orb(\vec{d})} \Psi_{\vec{d}}(x;
 \mathcal{O}_i).
\label{eq:phiorbit}
\end{equation}
The degree of each $\Psi_{\vec{d}}(x ; \cO_i) \in \Z[x]$ is 
$\tilde{\phi}(\lcm(d_1,\dots,d_q)) 2^{-
 \tilde{\beta}_0(\Gamma(V_{\vec{d}}))}$. 
\end{cor}
The polynomials $\Psi_{\vec{d}}(x ; \cO_i)$ may 
themselves be reducible. 
A complete characterization of the irreducibility of
$\Psi_{\vec{d}}(x ; \cO_i)$ remains open so far. 
Further discussion of this point is given in
Section~\ref{sec:polynomials}. 

This paper is organized as follows. 
In \sref{sec:eigenvalues}, we compute the eigenvalues of adjacency operators and Laplacians defined on the periodic cubical lattice. 
In \sref{sec:zeta}, we recall the results for zeta functions of
hypergraphs, and we present the results for cubical complexes in the cases
$d=q,q-1$. 
In \sref{sec:polynomials}, we prove \tref{thm:numoforbits}
and give a condition for $\Psi_{\vec{d}}(x ;\cO)$ 
to be irreducible (\pref{proposition:injective}). 
Also, we give further discussion of factorization and
present some observations for the case $q=2$. 

\section{Eigenvalues for periodic cubical complexes}\label{sec:eigenvalues}

\subsection{Eigenvalues of adjacency matrices}

The set $Y_d = \sqcup_{\sigma \in K_{d-1}} Y_{\sigma}$ of 
$d$-cubes (i.e., $|\sigma|=d$)
can be 
regarded as the set $K_{d-1} \times \Z_{n_1} \times\cdots \times \Z_{n_q}$ 
if we identify the intervals $I_i=[v_i, v_i+1] \in \cI_i^{\circ}$ 
and $I_i=[v_i, v_i]  \in \cP_i^{\circ}$ with the points $v_i \in \Z_{n_i}$ as follows:   
\begin{align*}
Y_d &= \bigsqcup_{\sigma \in K_{d-1}} \{I_1\times\dots\times I_q \subset \T^q \colon
I_i\in\cI_i^{\circ}~{\rm if~}  i\in\sigma;~
I_i\in\cP_i^{\circ}~{\rm otherwise} 
\}\\
 &\cong \{(\sigma, v_1,\dots,v_q) : \sigma \in K_{d-1}, \
 \text{$v_i \in \Z_{n_i}$ for $i \in K_0$}
\}.  
 \label{eq:kforms}
\end{align*}
Here, note that $|Y_d| = {q \choose d}|\vec{n}|$.

We write $V := \Z_{n_1} \times \cdots \times \Z_{n_q}$. 
For $d=0,1,\dots,q$, the space $C^d(Y)$ of $d$-cochains on
$Y$ is regarded as 
$C^d(K_{d-1} \times V)$, the space of functions on $K_{d-1} \times V$, 
under the identification introduced above. 
We continue to make this identification below although we use the
notation $C^d(Y)$. 

Now, we define the incidence operator $M_d : C^{d}(Y) \to C^{d+1}(Y)$ by 
\begin{align*}
M_d f(\eta, \v) 
&= \sum_{j \in \eta} 
\{
f(\eta_j, \v) + f(\eta_j, S_j \v)
\} \quad \text{for $\eta \in K_{d}$}, 
\end{align*}
where $\eta_j = \eta \setminus \{j\}$, 
$\v = (v_1,\dots,v_q) \in V$ and $S_j \v = (v_1,\dots, v_j+1, \dots, v_q)$. 
Equivalently, we can write $M_d$ as 
\[
M_d f(\eta, \v) 
= \sum_{\sigma \in K_{d-1} : \sigma \subset \eta} 
\{
f(\sigma, \v) + f(\sigma, S_{\eta \setminus \sigma} \v)
\} \quad \text{for $\eta \in K_{d}$}. 
\]
The matrix representation of $M_d$ is nothing but the incidence matrix 
between $Y_d$ and $Y_{d-1}$. 
The dual operator $M_d^* : C^{d+1}(Y) \to C^{d}(Y)$ of
$M_d$ with respect to the inner product
\begin{equation}
\bra f, g \ket_{C^d(Y)} :=
|V|^{-1} \sum_{\sigma \in K_{d-1}} \sum_{\v \in V} f(\sigma, \v)
\overline{g(\sigma, \v)}
\label{eq:innerproduct} 
\end{equation}
is given by 
\[
M_d^* f(\sigma, \v) 
= \sum_{\eta \in K_{d} : \eta \supset \sigma} 
\{
f(\eta, \v) + f(\eta, S_{\eta \setminus \sigma}^{-1} \v)
\} \quad \text{for $\sigma \in K_{d-1}$}. 
\]

Let $\hat{\Z}_{n_j}$ be the character group of $\Z_{n_j}$.  
An element $z_j$ of $\hat{\Z}_{n_j}$ is expressed 
as $z_j = \exp(2\pi \sqrt{-1} k_j /n_j)$ for some $k_j \in \Z_{n_j}$. 
We write $\hat{V} := \hat{\Z}_{n_1} \times \cdots \times \hat{\Z}_{n_q}$. 
For $\z = (z_1, \dots, z_q) \in \hat{V}$, 
we consider a subspace $C_{\z}^d$ of $C^d(Y)$ defined by 
\[
 C_{\z}^d := C_{\z}^d(Y) 
:=  \{f \in C^d(Y) : S_j f = z_j f 
\ \text{for any $j \in K_0$}\},  
\]
where $S_j f(\sigma, \v) := f(\sigma, S_j \v)$. 
It is clear that $\dim C_{\z}^d = |K_{d-1}| = {q \choose d}$, and 
$C_{\z}^d$ and $C_{\w}^d$ are orthogonal unless $\z = \w$. Hence, 
we have the orthogonal decomposition 
\[
 C^d(Y) = \bigoplus_{\z \in \hat{V}}  C^d_{\z}. 
\]

For $\z = (z_1, \dots, z_q) \in \hat{V}$, we also consider
the map 
$U_{\z} : C^d(K_{d-1}) \to C^d_{\z}(Y)$ defined by
\[
U_{\z}f(\sigma, \v) = f(\sigma) \z^{\v} \quad (\sigma \in K_{d-1}, \ \v \in V), 
\]
where
$\z^{\v} = \prod_{j=1}^q z_j^{v_j}$. 
Then, it is easy to see that $\bra U_{\z}f, U_{\z}g \ket_{C^d(Y)} 
= \bra f, g\ket_{C^d(K_{d-1})}$, i.e., 
$U_{\z}$ is unitary. 
The inverse map $U_{\z}^{-1} : C_{\z}^d(Y) \to C^d(K_{d-1})$ 
is the finite Fourier transform given by 
\[
(U_{\z}^{-1} f)(\sigma) 
= \frac{1}{|V|} 
\sum_{\v \in V} f(\sigma, \v) \z^{-\v} \quad \text{for $\sigma \in K_{d-1}$}. 
\]
Since $f(\sigma, \v) \in C^d_{\z}$, we see that 
$(U_{\z}^{-1} f)(\sigma) = f(\sigma, \zero)$. 

\begin{lem}
For $d=0,1,\dots,q-1$, 
the operator $M_d$ (also $M_d^*$) preserves each fiber of $\z$, i.e., 
 $M_d C_{\z}^d \subset C_{\z}^{d+1} $ (also
 $M_d^* C_{\z}^{d+1} \subset C_{\z}^{d}$). 
\end{lem}
\begin{proof}
It is easily verified. 
\end{proof}

The restriction $M_d$ on each fiber $\z$ is called the twisted operator of $M_d$ 
and denoted by $M_d(\z)$. Thus, we have the following direct
sum decomposition: 
\[
 M_d = \oplus_{\z \in \Vhat} M_d(\z) : \bigoplus_{\z \in \Vhat} C_{\z}^d
\to \bigoplus_{\z \in \Vhat} C_{\z}^{d+1}. 
\]
The situation for $M_d^*: C^{d+1}(Y) \to C^d(Y)$ is
similar.

\begin{lem}\label{lem:mhat}
For $\z \in \Vhat$, 
\begin{align*}
(M_d(\z) f)(\eta, \v)
 &= \sum_{\sigma \in K_{d-1} : \sigma \subset \eta}
 (1+z_{\eta \setminus \sigma}) f(\sigma, \v), \\ 
(M_d^*(\z)f)(\sigma,\v)
 &= \sum_{\eta \in K_{d} : \eta \supset \sigma} 
 (1+z_{\eta\setminus \sigma}^{-1}) f(\eta, \v).  
\end{align*}
\end{lem}
\begin{proof} By the definition of $M_d$ and 
$f \in C_{\z}^d$, we see that 
\begin{align*}
M_d(\z) f(\eta, \v) 
&= \sum_{\sigma \in K_{d-1} : \sigma \subset \eta} 
\{f(\sigma, \v) + f(\sigma, S_{\eta \setminus \sigma} \v)\} \\
&= \sum_{\sigma \in K_{d-1} : \sigma \subset \eta} 
(1 + z_{\eta \setminus \sigma}) f(\sigma, \v). 
\end{align*}
 The proof for $M_d^*(\z)$ is similar. 
\end{proof}

The operator $\tilde{M}_d(\z) := U_{\z}^{-1} M_d(\z) U_{\z} : C^d(K_{d-1}) \to C^{d+1}(K_{d})$ is unitarily equivalent to $M_d(\z)$ so that we abuse the notation $M_d(\z)$ for 
$\tilde{M}_d(\z)$ in what follows. 

The matrix representation of $M_d(\z)$ is the $|K_{d}| \times |K_{d-1}|$-matrix 
whose $(\eta, \sigma)$-element is given by 
$(1+z_{\eta \setminus \sigma}) \one(\sigma \subset \eta)$, 
where $\one(\sigma \subset \eta) = 1$ if $\sigma \subset \eta$; $0$ otherwise. 
Similarly, the matrix representation of $M_d^*(\z)$ is the $|K_{d-1}| \times |K_{d}|$-matrix  
whose $(\sigma, \eta)$-element is given by 
$(1+z^{-1}_{\eta \setminus \sigma}) \one(\sigma \subset \eta)$. 

Let $\aup_d = M_d^*M_d$ and $\adown_d = M_{d-1}M_{d-1}^*$, both acting on
  $C^d(Y)$. They also have the direct sum decompositions  
\[
 \aup_d = \oplus_{\z \in \Vhat} \aup_d(\z), \  
 \adown_d = \oplus_{\z \in \Vhat} \adown_d(\z), 
\] 
where $\aup_d(\z) = M_d^*(\z) M_d(\z)$
and $\adown_d(\z) = M_{d-1}(\z) M_{d-1}^*(\z)$ act on 
$C^d_{\z}$. 

In the following lemma, we use the metric on $K_{d-1}$ defined by 
$\rho(\sigma, \sigma') := |\sigma \setminus \sigma'| 
(= |\sigma' \setminus \sigma|)$. 
 \begin{lem}\label{lem:spec}
The $(\sigma,\sigma')$-element of the 
matrix representation of $\aup_d(\z)$ is given by 
\begin{align*}
a^{\mathrm{up}}_{\sigma \sigma'}(\z) 
&= 
\begin{cases}
\sum_{\eta \in K_d : \eta \supset \sigma}(1+z^{-1}_{\eta \setminus \sigma}) 
(1+z_{\eta \setminus \sigma}) 
& \text{if $\rho(\sigma,  \sigma')=0$ (i.e., $\sigma = \sigma'$)}, \\
(1+z_{\sigma' \setminus \sigma}^{-1}) 
(1+z_{\sigma \setminus \sigma'}) & \text{if $\rho(\sigma, \sigma')=1$}, \\
0 & \text{if $\rho(\sigma, \sigma') \ge 2$}, \\
\end{cases}
\end{align*}
and the set of eigenvalues of $\aup_d$ coincides with the union of the
  sets of eigenvalues of $\aup_d(\z)
  = (a^{\mathrm{up}}_{\sigma \sigma'}(\z))_{\sigma, \sigma' \in K_{d-1}}$
  for $\z \in \hat{V}$, i.e.,  
\[
  \spec(\aup_d)
  = \bigsqcup_{\z \in \hat{V}} \spec \Big(\aup_d(\z) \Big). 
\]
Similarly, 
the $(\sigma, \sigma')$-element of the matrix representation of $\adown_d(\z)$ 
is given by 
\begin{align*}
a^{\mathrm{down}}_{\sigma \sigma'}(\z) 
 &=
\begin{cases}
\sum_{\tau \in K_{d-2} : \tau \subset \sigma} 
(1+z_{\sigma \setminus \tau}) (1+z_{\sigma \setminus
 \tau}^{-1}) & \text{if $\rho(\sigma,  \sigma')=0$}, \\ 
(1+z_{\sigma \setminus \sigma'}) 
(1+z_{\sigma' \setminus \sigma}^{-1}) & \text{if $\rho(\sigma, \sigma')=1$}, \\
0 & \text{if $\rho(\sigma, \sigma') \ge 2$}, \\
\end{cases}
\end{align*}
and the set of eigenvalues of $\adown_d$ coincides with the union of the
  sets of eigenvalues of $\adown_d(\z)
  = (a^{\mathrm{down}}_{\sigma \sigma'}(\z))_{\sigma, \sigma' \in K_{d-1}}$
  for $\z \in \hat{V}$, i.e.,  
\[
  \spec(\adown_d)
  = \bigsqcup_{\z \in \hat{V}} \spec \Big(\adown_d(\z) \Big). 
\]
\end{lem}
\begin{proof}
For $\z \in \Vhat$, we have 
 \begin{align*}
  a^{\mathrm{up}}_{\sigma\sigma'}(\z) 
  &= \sum_{\eta \in K_{d}}
  (1+z_{\eta \setminus \sigma}^{-1})
  (1+z_{\eta \setminus \sigma'})
  \one(\eta \supset \sigma \cup \sigma'), \\
 a^{\mathrm{down}}_{\sigma \sigma'}(\z) 
&= 
\sum_{\tau \in K_{d-2}}    
(1+z_{\sigma \setminus \tau}) 
  (1+z_{\sigma' \setminus \tau}^{-1})
    \one(\tau \subset \sigma \cap \sigma'). 
 \end{align*}
 If $\rho(\sigma, \sigma') \ge 2$, then
 $\one(\eta \supset \sigma \cup \sigma')=
 \one(\tau \subset \sigma \cap \sigma')=0$.
 If $\rho(\sigma, \sigma') =1$, then $\eta$ (resp. $\tau$) must coincide
 with $\sigma \cup \sigma'$ (resp. $\sigma \cap \sigma'$) when 
 $\eta \supset \sigma \cup \sigma'$ 
(resp. $\tau \subset \sigma \cap \sigma'$).
We thus obtain the desired expressions. 
\end{proof}

Note that when $q \ge 2d-1$, which is equivalent to
${q \choose d} \ge {q \choose d-1}$, we have 
\[
 \spec \Big(\adown_{d}(\z) \Big) = 
 \spec \Big(\aup_{d-1}(\z) \Big) \cup \{0\},
\]
with the multiplicity of $0$ being ${q \choose d} - {q \choose d-1}$,
and similarly, 
when $q \le 2d - 1$, we have 
\[
 \spec \Big(\aup_{d-1}(\z) \Big) = 
 \spec \Big(\adown_{d}(\z) \Big) \cup \{0\},
\]
with the multiplicity of $0$ being ${q \choose d-1} - {q \choose d}$. 

\begin{cor}\label{cor:aqdown} 
 The eigenvalues of $\adown_q : C^q(Y) \to C^q(Y)$
  on the periodic $q$-cubical lattice $Y$ 
are $\{\sum_{j=1}^q 2\{1 + \cos (2 \pi k_j/n_j)\} : k_j \in \Z_{n_j}\}$.  
\end{cor}
\begin{proof}
From Lemma~\ref{lem:spec}, we know that 
$\adown_q(\z)$ is a scalar and equal to 
$\sum_{j=1}^q (1+z_j^{-1})(1+z_j)$. 
Therefore, the eigenvalues of $\adown_q$ are  
given by $\{\sum_{j=1}^q (1+z_j^{-1})(1+z_j) : z_j \in \hat{\Z}_{n_j}\}$.   
The assertion follows directly. 
\end{proof}

\subsection{Eigenvalues of Laplacians}

Techniques similar to those used in the previous section can
be applied to eigenvalue problems of Laplacians.  

For $\sigma, \tau \in K$ with $\tau \subset \sigma$ and $|\sigma
\setminus \tau|=1$, we write $\sgn(\sigma, \tau) = (-1)^{j-1}$ if $\sigma \setminus \tau$ is in the $j$th position 
of $\sigma$ in lexicographic order. 
For example, $\sgn(134, 34) = 1$, $\sgn(134, 14) = -1$ and $\sgn(134, 13) = 1$.  

For $d=0,1,\dots,q-1$,
let $\delta_d : C^d(Y) \to C^{d+1}(Y)$ be defined by 
\begin{align*}
 \delta_d f(\eta, \v)
 &= \sum_{\sigma \in K_{d-1} : \sigma \subset \eta} 
\sgn(\eta, \sigma) \{ f(\sigma, S_{\eta \setminus \sigma} \v) 
- f(\sigma, \v)\}
\quad \text{for $\eta \in K_{d}$}. 
\end{align*}
The dual operator of $\delta_d$,  $\delta_d^*$, is 
defined analogously to $M_d^*$ with respect to the inner product 
\eqref{eq:innerproduct}. 
As in the previous subsection, 
because both $\delta_d$ and $\delta^*_d$ preserve the fiber of $\z$, 
the operators $\delta_d$ and $\delta_d^*$ can be decomposed 
into the direct sums $\oplus_{\z \in \Vhat} \delta_d(\z)$
and $\oplus_{\z \in \Vhat} \delta_d^*(\z)$, respectively. 
As in the proof of \lref{lem:mhat}, 
we easily see that 
\begin{align*}
 \delta_d(\z) f(\eta, \v) 
 &=
 \sum_{\sigma \in K_{d-1} : \sigma \subset \eta}
 \sgn(\eta, \sigma) (z_{\eta \setminus \sigma} - 1) 
f(\sigma, \v) \quad \text{for $\eta \in K_{d}$}, \\
\delta_d^*(\z) f(\sigma, \v) 
 &= \sum_{\eta \in K_d : \eta \supset \sigma} 
\sgn(\eta,\sigma) (z_{\eta \setminus \sigma}^{-1} - 1) 
f(\eta, \v) \quad \text{for $\sigma \in K_{d-1}$}.
\end{align*}

\begin{lem}\label{lem:gtwice}
 For $d=0,1,\dots,q-1$, let $g : K_0 \to \C$ and 
define $G_d : C^d(Y) \to C^{d+1}(Y)$ by 
\[
G_df(\eta) = 
 \sum_{\sigma \in K_{d-1} : \sigma \subset \eta} 
\sgn(\eta, \sigma) g(\eta \setminus \sigma) f(\sigma) \quad (\eta \in K_d). 
\]
Then, $G_{d+1}G_d=0$. In particular, 
$\delta_{d+1}(\z) \delta_d(\z)=0$ for any $\z \in \hat{V}$. 
\end{lem}
\begin{proof}
For $f \in C^d(Y)$, 
\begin{align*}
(G_{d+1} G_d f)(\tau)
&= 
\sum_{\sigma \in K_{d} : \sigma \subset \eta}  
\sum_{\tau \in K_{d-1} : \tau \subset \sigma} 
 \sgn(\eta, \sigma) \sgn(\sigma, \tau) 
g(\eta \setminus \sigma) g(\sigma \setminus \tau) 
f(\tau) \\
&= 
\Big(\sum_{1 \le i < j \le q} g(i)g(j) 
\sum_{\{\eta \setminus \sigma, \sigma \setminus \tau\} = \{i,j\}}
 \sgn(\eta, \sigma) \sgn(\sigma, \tau) \Big) f(\tau) \\
&= 0. 
 \end{align*}
The last sum is taken separately for two cases,  
$(\eta \setminus \sigma, \sigma \setminus \tau)=(i,j)$ 
and $(j,i)$,  
which is equal to $0$ as usual. 
\end{proof}

From this lemma, we have the cochain complex 
\[
\cdots \to C^{d-1}_{\z} \stackrel{\delta_{d-1}(\z)}{\to} C^d_{\z} 
\stackrel{\delta_{d}(\z)}{\to} C^{d+1}_{\z} \to \cdots, 
\]
and hence the cohomology group 
$H^d_{\z} := \ker \delta_d(\z) / \im \delta_{d-1}(\z)$ is defined. 
\begin{lem}\label{lem:trivial-cohom}
For all $d=0,1,\dots, q-1$ and $\z \in \hat{V}$, 
the cohomology group $H^d_{\z}$ is trivial for $\z \not= \one$ and 
$H^d_{\one} = C_{\one}^d \simeq \C^{{q \choose d}}$. 
\end{lem}
\begin{proof}
First, we note that $H^d(Y)=\oplus_{\z \in \Vhat} H^d_{\z}$
 is the $d$th cohomology of the $q$-dimensional torus and is
 given by $H^d(Y)=\C^{q\choose d}$. For $\z=\vec{1}$, since
 $\delta_d(\z)=0$,  we have 
$H^d_{\vec{1}}=C^d_{\vec{1}}\simeq \C^{q\choose d}$. Hence, $H^d_{\z}=0$ for $\z\neq \vec{1}$.
\end{proof}

\begin{cor}\label{cor:kerneldim}
If $\z \not= \one$, then 
$\dim \ker \delta_d(\z)={q-1\choose d-1}$ and 
$\dim \ker \delta_{d-1}^*(\z) = {q-1\choose d}$ 
for $d=0,1,\dots,q$. Here, ${q-1 \choose -1} = {q-1 \choose q} = 0$. 
\end{cor}
\begin{proof} From the rank-nullity theorem together with \lref{lem:trivial-cohom}, we have  
\[
 {q \choose d} = \dim \ker \delta_d(\z) + \dim \ker
 \delta_{d+1}(\z) \quad (d=0,1,\dots,q-1). 
\]
Since $\dim \ker \delta_0(\z)=0$, taking the alternating sum above yields 
$\dim \ker \delta_d(\z) = {q-1 \choose d-1}$. 
Since $ \ker \delta_{d-1}^*(\z) 
= (\im \delta_{d-1}(\z))^{\perp} = (\ker
 \delta_d(\z))^{\perp}$, from \lref{lem:trivial-cohom}, we have 
\[
C^d_{\z} = \ker \delta_{d}(\z) \oplus \ker \delta_{d-1}^*(\z).   
\]
Therefore, $\dim \ker \delta_{d-1}^*(\z) = {q-1\choose d}$.  
\end{proof}

Next, we give the matrix representations for the up/down Laplacians. 
 Let $\lup_d = \delta_d^* \delta_d$
  and $\ldown_d = \delta_{d-1} \delta_{d-1}^*$. 
These operators can be decomposed as 
$\lup_d = \oplus_{\z \in \Vhat} \lup_d(\z)$ and 
$\ldown_d = \oplus_{\z \in \Vhat} \ldown_d(\z)$. 
As in the case of \lref{lem:spec}, we have the following. 

 \begin{lem}\label{lem:partialef}
The $(\sigma,\sigma')$-element of the 
matrix representation of $\lup_d(\z)$ is given by 
\[
\ell^{\mathrm{up}}_{\sigma \sigma'}(\z) = 
\sum_{\eta \in K_d : \eta \supset \sigma \cup \sigma'}
\sgn(\eta, \sigma)\sgn(\eta, \sigma')
(z_{\eta \setminus \sigma}^{-1}-1)(z_{\eta \setminus \sigma'}-1) 
\]
and the $(\sigma, \sigma')$-element of the matrix representation
  of $\ldown_d(\z)$ 
is given by 
\[
\ell^{\mathrm{down}}_{\sigma \sigma'}(\z) = 
\sum_{\tau \in K_{d-2} : \tau \subset \sigma \cap \sigma'} 
\sgn(\sigma, \tau) \sgn(\sigma', \tau)
(z_{\sigma \setminus \tau}-1) 
(z_{\sigma' \setminus \tau}^{-1}-1).  
\]
 \end{lem}

\begin{cor}\label{cor:upplusdown}
 Suppose $K$ is the complete simplicial complex over $\{1,2,\dots,q\}$.
 Then, we have 
 \[
\lup_d(\z) + \ldown_d(\z) = 
 \Big\{2q - \sum_{i=1}^q (z_i + z_i^{-1}) \Big\} I. 
\] 
\end{cor}
\begin{proof}
 In \lref{lem:partialef}, we chose $\sigma, \sigma' \in K_{d-1}$. 
If $\sigma \not= \sigma'$, then $\ell^{\mathrm{up}}_{\sigma \sigma'}(\z) 
= \ell^{\mathrm{down}}_{\sigma \sigma'}(\z) = 0$ 
unless $\rho(\sigma, \sigma') = 1$. 
When $\rho(\sigma, \sigma') = 1$, we see that 
$\tau = \sigma \cap \sigma'$, $\eta = \sigma \cup \sigma'$, 
$\sigma \setminus \tau = \eta \setminus \sigma'$ and 
$\sigma' \setminus \tau = \eta \setminus \sigma$. 
Using the same argument as in the proof of
 \lref{lem:gtwice}, we obtain 
\[
  \sgn(\sigma \cup \sigma', \sigma) \sgn(\sigma \cup \sigma', \sigma')
+ \sgn(\sigma, \sigma \cap \sigma') \sgn(\sigma', \sigma \cap \sigma') 
=0, 
\]
and hence 
$\ell^{\mathrm{up}}_{\sigma \sigma'}(\z) +  \ell^{\mathrm{down}}_{\sigma \sigma'}(\z) = 0$. 
If $\sigma = \sigma'$, then we have 
\begin{align*}
\lefteqn{\ell^{\mathrm{up}}_{\sigma\sigma}(\z) +  \ell^{\mathrm{down}}_{\sigma \sigma}(\z)} \\ 
 &= \sum_{\eta \in K_d : \eta \supset \sigma} 
 (z_{\eta \setminus \sigma}^{-1}-1)(z_{\eta \setminus \sigma}-1) 
+ \sum_{\tau \in K_{d-2} : \tau \subset \sigma} 
(z_{\sigma \setminus \tau}-1) (z_{\sigma \setminus \tau}^{-1}-1) \\
&= \sum_{i=1}^q (z_i-1)(z_i^{-1}-1). 
\end{align*}
This completes the proof. 
\end{proof}

\begin{lem}\label{lem:ev_Ldz}
 The eigenvalues of $\lup_d(\z)$ (resp. $\ldown_d(\z)$) on
 the periodic $q$-cubical lattice $Y$ are  
$2q - \sum_{i=1}^q (z_i + z_i^{-1})$ with multiplicity ${q-1 \choose
 d}$ (resp. ${q-1 \choose d-1}$)
 and $0$ with multiplicity ${q-1 \choose d-1}$ (resp. ${q-1 \choose d}$).  
\end{lem}
 \begin{proof}
By \cref{cor:upplusdown}, for any $f \in \ker \ldown_d(\z)$, we have 
\[
  \lup_d(\z) f = 
\big\{2q - \sum_{i=1}^q (z_i + z_i^{-1}) \big\} f. 
\]
Also, for any non-zero $f \in \ker \lup_d(\z)$, by definition, $\lup_d(\z) f = 0$.
  Since $\dim \ker \ldown_d(\z) = \dim \ker \delta_{d-1}^*(\z) = {q-1 \choose d}$ 
and $\dim \ker \lup_d(\z) = \dim \kernel \delta_d(\z) = {q-1 \choose d-1}$ 
by \cref{cor:kerneldim}, the space $C^d_{\z}(Y)$ is spanned by those eigenfunctions. 
 \end{proof}

\begin{cor}\label{cor:ev_Ld}
The eigenvalues of $\lup_d$ are given by 
\begin{equation}
 \spec(\lup_d) 
= \left\{ 2q - 2\sum_{i=1}^q \cos \frac{2 \pi k_i}{n_i} : k_i \in \Z_{n_i}\right \} \cup \{0\}. 
\label{eq:specoflupd} 
\end{equation}
The multiplicity of the eigenvalue $2q - 2\sum_{i=1}^q \cos \frac{2
 \pi k_i}{n_i}$ for each $\vec{k} =(k_1,\dots, k_q)$
 is ${q-1 \choose d}$, and that of $0$ is ${q-1 \choose
 d-1} |\vec{n}|$. 
\end{cor}

Note that the eigenvalue for $\vec{k} = (0,0,\dots,0)$ with multiplicity 
${q-1 \choose d}$ is also $0$, 
but this eigenvalue is represented within the first set of
the right-hand side of \eqref{eq:specoflupd}, not the second set.

\section{Zeta functions of periodic cubical complexes}\label{sec:zeta}

The zeta functions of cubical complexes can be reformulated as those of
hypergraphs (see Definition~\ref{def:zetafn2}).  

A hypergraph is a pair $H=(V,E)$ of disjoint sets, where 
$V$ is a non-empty set and the elements of $E$ are non-empty subsets of $V$. 
An element of $V$ (resp. $E$) is called a hypervertex (resp. hyperedge). 
A hypervertex $v \in V$ is said to be
incident to $e \in E$ if $v$ is included in $e$. 
The $|V|\times |E|$-matrix $M$ indexed by the elements of
$V$ and $E$ is defined as 
$M_{v,e} = 1$ if $v$ is incident to $e$; 0 otherwise. This is called
the incidence matrix for $H$. The degree, $\deg(v)$, of a hypervertex $v$ is the number of hyperedges that include $v$, and 
the degree, $\deg(e)$, of a hyperedge $e$
is the number of hypervertices that are included in $e$. 
A hypergraph $H = (V,E)$ is said to be $(a,b)$-regular if $\deg(v) = a$ for all $v \in V$ 
and $\deg(e) = b$ for all $e \in E$.   

\begin{ex}\label{ex:hypergraph}
(1) When $\deg(e)=2$ for every $e \in E$, then a hypergraph is nothing but a graph. \\
(2) A simplicial complex over a set $V$ can be viewed as a hypergraph by regarding 
all simplices as $E$.  \\
(3) For a $q$-dimensional cubical complex, let $V$ be the set of $(d-1)$-cubes and 
 $E$ the set of $d$-cubes. Then $H=(V,E)$ forms a hypergraph.
 In particular, $H = (V,E)$ is a $(2(q-d+1), 2d)$-regular hypergraph 
 if $V = Y_{d-1}$ and $E = Y_d$ for
 the periodic $q$-cubical lattice $Y$. 
\end{ex}

A closed path in $H$ is a sequence such that 
$c = (v_0, e_0, v_1, e_1, \dots, v_{n-1}$, $e_{n-1})$, where $v_{i+1} \in e_{i} \cap e_{i+1}$ 
and $v_i \not= v_{i+1}$ for all $i \in \Z_n$. Note that $v_0 \in e_{n-1} \cap e_0$. 
The length of $c$ is the number of hyperedges in $c$, denoted by $|c|$. 
We say that a closed path $c$ is a closed geodesic if $e_i \not= e_{i+1}$ for all $i \in \Z_{|c|}$.  
We denote by $c^m$ the $m$-multiple of a closed geodesic $c$
formed by $m$ repetetions of $c$. 
If a closed geodesic $c$ is not expressed as 
an $m$-multiple of a closed geodesic with $m \ge 2$, then $c$ is said to be prime. 
Two prime closed geodesics are said to be equivalent if one
is obtained from the other through a cyclic permutation. 
An equivalence class of prime closed geodesics is called a prime cycle. 
The length of a prime cycle $\primitive$ is defined as
the length of a representative and is denoted by $|\primitive|$. 
The (Ihara) zeta function of a finite hypergraph is defined as follows. 

\begin{defn}\label{def:zetafn2}
 For $u \in \C$ with $|u|$ sufficiently small, the (Ihara) zeta function of a finite hypergraph 
$H$ is defined by 
\[
 \zeta_H(u) = \prod_{\primitive \in P} (1 - u^{|\primitive|})^{-1}, 
\]
where $P$ is the set of prime cycles of $H$. 
\end{defn}

The factorization theorem for zeta functions of finite graphs was obtained by 
H.~Bass \cite{bass} and a conceptually simpler proof employing oriented linegraph structure
was given by Kotani-Sunada \cite{kotani-sunada}. 

One can associate with a hypergraph $H=(V,E)$ a bipartite graph $B_H$ whose vertex partite sets are 
$V$ and $E$, and the incidence relation gives edges in $B_H$, i.e., 
$V(B_H) = V \sqcup E$ and every edge in $E(B_H)$ connects a
vertex in $V$ to one in $E$;  
$v \in V$ and $e \in E$ are joined when $v \in e$. 
The definition of prime cycles given above fits for the cycle structure of
$B_H$ when $\deg(v) \ge 2$ for all $v \in V$ as discussed in
\cite{storm}, and hence the theorem for graphs can be
extended to the hypergraph setting as follows: 

\begin{thm}[\cite{storm}]\label{thm:storm}
Let $H = (V,E)$ be a finite, connected hypergraph such that
 $\deg(v) \ge 2$ for all $v \in V$ with adjacency matrix $A$
 and diagonal degree matrix $D$ in $B_H$. 
Then, 
\[
 \zeta_{H}(u)
= (1-u)^{\chi(B_H)} \det(I - \sqrt{u} A + u Q)^{-1}, 
\]
 where $I$ is the $m \times m$ identity matrix with $m = |V|+|E|$,
 $Q = D-I$, 
$B_H$ is the bipartite graph associated with $H$, and 
$\chi(B_H) = |V|-|E|$ is the Euler characteristic of $B_H$. 
\end{thm}

In the above theorem, although $\sqrt{u}$ appears in the expression,  
the zeta function is a rational function of $u$,  
because the length of the corresponding cycle in $B_H$ of a cycle in $H$ is doubled. 
The colored, oriented linegraph is constructed from $H$ in the proof given in \cite{storm}, 
which is based on an idea presented in \cite{kotani-sunada}.

The following theorem can also be regarded as a restatement 
of Hashimoto's theorem on zeta functions of semi-regular graphs because
the bipartite graph $B_H$ is $(a,b)$-semi-regular when $H$ is
$(a,b)$-regular. The number of vertices in $B_H$ is $|V|+|E|$ and
that of edges in $B_H$ is $a|V| = b|E|$. 
We can define the dual hypergraph $H^*$ of $H=(V,E)$
by interchanging the roles of
partite sets in $B_H$, i.e., $H^* = (E, E_V)$ with the incidence relation
determined by $B_H$, where $E_V := \{A_v \subset E: v \in V\}$ with 
$A_v := \{e \in E : e \ni v \}$ for $v \in V$. 
Clearly, $H^*$ is $(b,a)$-regular if $H$ is
$(a,b)$-regular. 
Also, it is clear that $\zeta_H(u) = \zeta_{H^*}(u)$ because
the cycle structures in $H$ and $H^*$ are identical 
since $B_H = B_{H^*}$ as graphs. 

\begin{thm}[\cite{storm}]\label{thm:storm2}
Let $H = (V,E)$ be a finite connected $(a,b)$-regular hypergraph with
 $a, b \ge 2$, and  
let $M$ be the incidence matrix between $V$ and $E$. 
Let $\alpha = a-1$ and $\beta  = b-1$. Then, 
\begin{align*}
\zeta_{H}(u)^{-1}
&= (1-u)^{-\chi(B_H)} (1+\beta u)^{|E| - |V|} \det\big( (1+\alpha u)(1+\beta u)I_{V} - u MM^* \big) \\
&= (1-u)^{-\chi(B_H)} (1+\alpha u)^{|V| - |E|} \det\big(
 (1+\alpha u)(1+\beta u)I_{E} - u M^*M \big),
\end{align*} 
where $\chi(B_H) = |E| - \alpha|V| = |V| - \beta |E|$, and 
$I_V$ (resp. $I_E$) is the $|V|\times |V|$ (resp. $|E| \times
 |E|$) identity matrix. 
\end{thm}
\begin{proof}
In \tref{thm:storm}, the adjacency matrix $A$ and the diagonal matrix
 $Q$ are expressed by 
$A = \begin{pmatrix} 
      O & M \\
     M^* & O
\end{pmatrix}$ and  
$Q = \begin{pmatrix} 
      \alpha I_V & O \\
     O & \beta  I_E
\end{pmatrix}$. 
Using the determinantal identity 
\begin{equation}
 \det  \begin{pmatrix} 
      P_{11} & P_{12} \\
     P_{21} & P_{22}
\end{pmatrix}
= \det P_{11} \cdot \det(P_{22} - P_{21}P_{11}^{-1}P_{12}), 
\label{eq:det}
\end{equation}
we obtain 
\begin{align*}
\lefteqn{\det(I - \sqrt{u} A + uQ)}\\
&= 
(1+\alpha u)^{|V|} \cdot (1+\alpha u)^{-|E|}
\det \Big( (1+\alpha u)(1+\beta u)I_E - u M^* M \Big). 
\end{align*}
The second equality is obtained in the same way by changing the roles of $1$ and $2$ in 
\eqref{eq:det}. 
\end{proof}

\tref{thm:storm2} together with \cref{cor:aqdown} yields 
the following theorem. 

\begin{thm}\label{thm:zetaforcubical}
Let $Y$ be a $d$-dimensional cubical lattice with side length $\vec{n} =
 (n_1,n_2,\dots,n_q)$ and $Y^{(d)}$ its $d$-skeleton. 
Let $\aup_{d-1}(\z)$ (resp. $\adown_{d}(\z)$) on
 $C_{\z}^{d-1}$ (resp. $C_{\z}^d$) be the twisted adjacency operator.  
 Then, 
 \begin{align*}
\zeta_{Y^{(d)}}(u)^{-1} 
 &= (1-u)^{\kappa_d |\vec{n}| } (1+\beta_du)^{\gamma_d|\vec{n}|} \\
 &\quad \times \prod_{\z \in \Vhat}
 \det\Big( (1+\alpha_du)(1+\beta_du)I_{K_{d-2}} - u \aup_{d-1}(\z)\Big) \\ 
  &= (1-u)^{\kappa_d |\vec{n}|}
  (1+\alpha_du)^{-\gamma_d|\vec{n}|} \\
 &\quad \prod_{\z \in \Vhat} \det\Big( (1+\alpha_du)(1+\beta_du)I_{K_{d-1}}
 - u \adown_d(\z) \Big), 
 \end{align*}
 where
 $\alpha_d = 2q - 2d + 1$, $\beta_d = 2d-1$, $\gamma_d = {q \choose d} - {q \choose
 d-1}$ and 
 $\kappa_d = (q-d){q \choose d-1} +(d-1) {q \choose d}$.  
\end{thm}
\begin{proof}
 We consider the case in which $V = Y_{d-1}$ and $E = Y_d$. 
Then $H=(V,E)$ is $(2(q-d+1),2d)$-regular as in
 \eref{ex:hypergraph}(3) and 
$-\chi(B_H) = \kappa_d |\vec{n}|$.  
 Since the incidence matrix $M$ in \tref{thm:storm2} is set to be
 $\oplus_{\z \in \Vhat} M_{d-1}(\z)$,
we have $M_{d-1}(\z)M^*_{d-1}(\z) = \adown_{d}(\z)$, and
 hence  
\begin{align*}
 \zeta_{Y^{(d)}}(u)^{-1} 
&= (1-u)^{\kappa_d|\vec{n}|} 
 (1 + \alpha_du)^{|Y_{d-1}| - |Y_d|} \\
 &\quad \times
\prod_{\z \in \Vhat} 
\det 
\Big\{(1 + \alpha_du)(1+\beta_du)I_{K_{d-1}} - u
 \adown_{d}(\z) \Big\}, 
\end{align*}
 where $\alpha_d=2(q-d)+1$, $\beta_d = 2d-1$.
 The second equality is obtained similarly. 
This completes the proof. 
\end{proof}

\begin{rem}\label{rem:symmetry}
(1) The former expression is useful for $q \ge 2d-1$, while
the latter is useful for $q \le 2d-1$ as these two conditions are
equivalent to $\gamma_d \ge 0$ and $\gamma_d \le 0$,
 respectively. \\
 (2) The map $d \mapsto q-d+1$ leaves $\alpha_d, \beta_d$
 and $\kappa_d$ invariant, and $\gamma_{q-d+1} = -
 \gamma_d$. It follows from this invariance 
 that $\zeta_{Y^{(q-d+1)}}(u) = \zeta_{Y^{(d)}}(u)$. 
\end{rem}

We obtain \tref{thm:zeta} as a special case of \tref{thm:zetaforcubical}. 

\begin{proof}[Proof of \tref{thm:zeta}]
 We consider the case $d=q$, i.e., $V = Y_{q-1}$ and $E =
 Y_q$. Then $H=(V,E)$ is $(2,2q)$-regular as in \eref{ex:hypergraph}(3).
 In this case, $\adown_q(\z) = M_{q-1}(\z) M_{q-1}^*(\z)$ is
 a scalar, and we have 
$\adown_q(\z) = \sum_{i=1}^q 2(1+ \cos \theta_i)$ from \cref{cor:aqdown}, 
 where $\theta_i = 2\pi k_i /n_i$.
 Then, from \tref{thm:zetaforcubical}, we have 
\begin{align*}
 \zeta_Y(u)^{-1} 
&= (1-u)^{\kappa_q|\vec{n}|}
 (1 + \alpha_q u)^{\gamma_q |\vec{n}|} \\
 &\quad \times
\prod_{k_1=1}^{n_1} \cdots \prod_{k_q=1}^{n_q}
\Big\{(1 + \alpha_q u)(1+\beta_q u) - u \sum_{i=1}^q 2(1+
 \cos \theta_i) \Big\}, 
\end{align*}
 where $\alpha_q=1$, $\beta_q = 2q-1$ and $\kappa_q = \gamma_q = q-1$.
Therefore, we obtain 
\[
 \zeta_Y(u)^{-1} 
 = (1-u^2)^{(q-1)|\vec{n}|} 
 \prod_{k_1=1}^{n_1} \cdots \prod_{k_q=1}^{n_q}
\Big\{1 - 2u \sum_{i=1}^q \cos \theta_i + (2q-1)u^2 \Big\}. 
\]
This completes the proof. 
\end{proof}

We denote by $e_k(\t)$ the $k$th elementary symmetric
polynomial of $\t =(t_1,\dots,t_q)$ defined by the expansion formula 
\begin{equation}
\prod_{k=1}^q (\la+t_k) = \sum_{k=0}^q e_k(\t) \la^{q-k}. 
\label{eq:elementary_symmetric} 
\end{equation}

\begin{prop}
Let $w_i = 2 + z_i + z_i^{-1}$. Then, 
\[
 \det(t - u \aup_1(\z)) 
= \sum_{k=0}^q (2-k)2^{k-1} e_k(\w) (t- u e_1(\w))^{q-k} u^k. 
\]
\end{prop}
\begin{proof}
From \lref{lem:spec}, we see that the $(j,j)$-element of the matrix 
$\aup_1(\z) - e_1(\w) I_{q}$ is given by 
\[
\sum_{k \not=j} (1+z_k^{-1})(1+z_k) - \sum_{k=1}^q (1+z_k^{-1})(1+z_k^{-1})
= - (1+z_j^{-1})(1+z_j). 
\]
Thus we have 
\[
\aup_1(\z) - e_1(\w) I_{q} = D_{\one + \z^{-1}} (J_q - 2I_q)
 D_{\one + \z}, 
\]
where $J_q$ is the $q \times q$ matrix whose elements are
 all $1$, 
$D_{\one + \z}$ (resp. $D_{\one + \z^{-1}}$) 
is the diagonal matrix whose $(j,j)$-element is $(1+z_j)$ 
(resp. $(1+z_j^{-1})$). 
Hence, setting $\la = t - e_1(\w)$, we obtain 
\begin{align*}
 \det(t - \aup_1(\z)) 
&= \det (\la  - D_{\one + \z^{-1}} (J_q - 2 I_q) D_{\one + \z}) \\
&= \det (\la - (J_q - 2 I_q) D_{\w}) \\
&= \det D_{\w} \det\big( (\la D_{\w^{-1}}+ 2I_q) - J_q \big).  
\end{align*}
It is easy to see that  for any $\aa = (a_1,\dots,a_q)$ the
 relation 
\[ 
\det (D_{\aa} - J_q) = e_q(\aa) - e_{q-1}(\aa)
\]
holds. Thus we obtain   
\begin{align*}
\det(t - \aup_1(\z)) 
&= \left( 1-\frac{q}{2} + \frac{\la}{2} \frac{\partial}{\partial \la} \right) \prod_{i=1}^q (\la + 2 w_i) \Big|_{\la = t - e_1(\w)}. 
\end{align*}
Expanding the right-hand side in $\la$ and using 
\eqref{eq:elementary_symmetric}, we reach the desired
 relation. 
\end{proof}

\begin{proof}[Proof of the latter half of \tref{thm:zeta}]
A simple calculation shows that 
 \begin{align*}
 \zeta_{Y^{(2)}}(u)^{-1} 
 &= (1-u)^{\kappa_2 |\vec{n}| } (1+3u)^{\gamma_2|\vec{n}|}
\times \prod_{\z \in \Vhat} F_1^{\mathrm{up}}(u, \z), 
 \end{align*}
where $\kappa_2 = q(3q-5)/2$, $\gamma_2 = q(q-3)/2$ and 
\[
F_1^{\mathrm{up}}(u, \z) = \sum_{k=0}^q 
(2-k)2^{k-1} e_k(\w) 
 \Big( 1 - u \sum_{i=1}^q (z_i + z_i^{-1}) + 3(2q-3) u^2 \Big)^{q-k} u^k.  
\]
\end{proof}

We remark that the right-hand sides of the 
two expressions in \tref{thm:zeta} must coincide for $q=2$ by symmetry. 
Indeed, when $q=2$, we see that $\kappa=1$, $\gamma = -1$ and 
$F_1^{\mathrm{up}}(u, \z) = 
(1+u)(1+3u)\big(1 - u \sum_{i=1}^2 (z_i + z_i^{-1}) + 3 u^2
\big)$. Thus, in this case, 
both of these are equal to $(1-u^2)^{|\vec{n}|} 
\prod_{k_1=1}^{n_1}\prod_{k_2=1}^{n_2}\big(1 - 2u \sum_{i=1}^2 \cos 2 \pi k_i/n_i + 3 u^2 \big)$.

\section{Cyclotomic-like polynomials}
\label{sec:polynomials}

 For $\vec{d} = (d_1,\dots,d_q) \in \N^q$, 
 we define the following polynomial in $x$: 
\begin{equation}
 \Psiprime_{\vec{d}}(x) := 
\prod_{j_1 \in J_{d_1}} \cdots \prod_{j_q \in J_{d_q}} \Big(x -
2\sum_{i=1}^q \cos \frac{2\pi j_i}{d_i}\Big),  
\label{eq:cyclotomic-like}
\end{equation}
With the notation used in Introduction, this would be
written $\Psi_{\vec{d}}(x,1)$, but here, we use the above
more concise notation. 
Since the homogeneous polynomial $\Psi_{\vec{d}}(x,y)$ 
can be recovered from $\Psi_{\vec{d}}(x)$, hereafter we
focus on $\Psi_{\vec{d}}(x)$. 
We note that $ \Psiprime_{\vec{d}} =  \Psiprime_{\vec{d}'}$ 
when $\vec{d}'$ is a permutation of $\vec{d}$ and 
that the degree of $\Psi_{\vec{d}}$ is equal to $\prod_{i=1}^q
\tilde{\phi}(d_i)$. 
For $q=1$, there is a known explicit form of
$\Psiprime_d(x)$ 
that can be obtained using cyclotomic polynomials.

\begin{lem}\label{lem:Psid}
Suppose $q=1$. 
Then $z^{\tilde{\phi}(d)}\Psiprime_{d}(z+z^{-1})$
is the $d$th cyclotomic polynomial for $d \ge 3$ and 
the square of that for $d=1,2$. 
Moreover, $\Psi_d(x)$ is irreducible for any $d \in \N$. 
\end{lem}
\begin{proof}
The proof is trivial for $d=1,2$. Suppose $d \ge 3$. 
It is easily seen that 
$z^{\tilde{\phi}(d)}\Psiprime_{d}(z+z^{-1})$ is a monic
 polynomial in $z$ of degree $2 \tilde{\phi}(d)$, 
since the degree of $\Psiprime_{d}(x) \in \Z[x]$ is $\tilde{\phi}(d)$.
This polynomial has the following $\varphi(d)= 2 |J_d|$ distinct roots
\[
\{ z= \exp(2\pi \sqrt{-1} j/d) \mid j \in J_d \}
\cup
\{ z= \exp(-2\pi \sqrt{-1} j/d) \mid j \in J_d \}.
\]
These properties characterize the cyclotomic polynomial.
If $\Psi_d(x)$ is reducible, then so is $\Psi_d(z +
 z^{-1})$. This contradicts the irreducibility of cyclotomic
 polynomials. 
\end{proof}

\begin{ex}\label{ex:table}
$\Phi_d(x)$ is the $d$th cyclotomic polynomial, and 
$\Psi_d(x)$ is that defined in \eqref{eq:cyclotomic-like} for $q=1$. 
We have the following explicit forms: 
\begin{center}
\begin{minipage}{7cm}
\begin{align*}
\Phi_1(x) &= x-1 \\
\Phi_2(x) &= x+1 \\
\Phi_3(x) &= x^2+x+1 \\
\Phi_4(x) &= x^2+1 \\ 
\Phi_5(x) &= x^4+x^3+x^2+x+1 \\ 
\Phi_6(x) &= x^2-x+1 \\
\Phi_7(x) &= x^6+x^5+x^4+x^3+x^2+x+1 \\
\Phi_8(x) &= x^4+1 \\
\Phi_9(x) &= x^6+x^3+1\\
\Phi_{10}(x) &= x^4-x^3+x^2-x+1 
\end{align*}
\end{minipage}
\begin{minipage}{5cm}
\begin{align*}
\Psi_1(x) &= x-2\\
\Psi_2(x) &= x+2\\
\Psi_3(x) &= x+1 \\
\Psi_4(x) &= x \\ 
\Psi_5(x) &= x^2+x-1\\
\Psi_6(x) &= x-1 \\
\Psi_7(x) &= x^3+x^2-2 x-1 \\
\Psi_8(x) &= x^2-2 \\
\Psi_9(x) &= x^3-3x+1 \\
\Psi_{10}(x) &= x^2-x-1 
\end{align*}
\end{minipage}
\end{center}

\noindent
Suppose that $d_1=1,2,3,4$ or $6$. 
Then, for $\vec{d'}=(d_2,\ldots,d_q)$, 
we have $ \Psiprime_{\vec{d}}(x) =
 \Psiprime_{\vec{d}'}(\Psi_{d_1}(x))$ since $J_{d_1} =
 \{1\}$. 
From this observation, one can compute several 
$\Psi_{\vec{d}}$s by using the above table. 
\end{ex}

In \tref{thm:zeta}, the polynomial 
\begin{equation}
F_{\vec{n}}(x,y) 
:= \prod_{k_1=1}^{n_1} \cdots 
\prod_{k_q=1}^{n_q} 
\Big(x -2y\sum_{i=1}^q \cos \frac{2\pi k_i}{n_i}\Big)
\label{eq:product}
\end{equation}
appears. This polynomial can be decomposed in terms of $\Psi_{\vec{d}}(x,y)$ as 
$x^n-1$ is decomposed into a product of cyclotomic polynomials $\Phi_d(x)$. 

\begin{proof}[Proof of \tref{thm:zeta2}]
Let $\tilde{J}_d = \{j \in \{1,2,\dots,d\} : \gcd(j,d)=1\}$. 
We note that 
\begin{align*}
\big\{\frac{k}{n} \ | \ k \in \{1,2,\dots, n\} \big\}
&= 
\bigsqcup_{d | n} 
\big\{ \frac{j}{d} \ | \ j \in \tilde{J}_d \big\}
\end{align*}
and each set on the right-hand side can be further decomposed as
\[
 \big\{ \frac{j}{d} \ | \ j \in \tilde{J}_d \big\}
= \begin{cases}
\big\{ \frac{j}{d} \ | \ j \in J_d\}
   \sqcup \big\{ \frac{d-j}{d} \ | \ j \in J_d\}  &
   \text{for $d \ge 3$}, \\
\big\{ \frac{j}{d} \ | \ j \in J_d\} & \text{for $d =1,2$}. 
  \end{cases}
\]
Therefore, we have 
\begin{align*}
F_{\vec{n}}(x,y) &= 
\prod_{d_1|n_1}\!
\cdots \prod_{d_q|n_q}\!
\prod_{j_1 \in J_{d_1}}\! \cdots \!\prod_{j_q \in J_{d_q}} 
\Big(x -2y\sum_{i=1}^q \cos \frac{2\pi
 j_i}{d_i}\Big)^{\epsilon(d_1) \times \cdots \times \epsilon(d_q)}\\
&= 
\prod_{d_1|n_1}\!
\cdots \!\prod_{d_q|n_q}
\Psi_{\vec{d}}(x,y)^{\epsilon(\vec{d})}, 
\end{align*}
where $\epsilon(d) = 2$ for $d \ge 3$ and $\epsilon(d)=1$ for $d =1,2$. 
This completes the proof 
since $\zeta_Y(u)^{-1} = (1-u^2)^{(q-1)|\vec{n}|}
 F_{\vec{n}}(1 + (2q-1)u^2, u)$. 
\end{proof}

Now, in order to factorize $\Psi_{\vec{d}}(x)$ further, 
we consider the orbit structure for Galois actions. 

Suppose $d_i \ge 3$ for $i =1,\ldots,q$.
Let $N=d_1\cdots d_q$.
We identify $J_d$ for $d\ge 3$ with the set of representatives of
$(\Z/d \Z)^\times$ modulo $x \mapsto -x$.
We identify $(\Z/N\Z)^\times $
with $\{ m \in \N \mid m < N, \gcd(m,N)=1\}$.
The group $(\Z/N\Z)^\times$ acts on
$J_{d_1} \times \cdots \times J_{d_q}$ as component-wise multiplication:
\begin{equation}
(j_1,\ldots, j_q) \mapsto (a j_1,\ldots, a j_q)
\label{equation:GaloisAction}  
\end{equation}
for $a \in (\Z/N\Z)^\times$.

Before proving \tref{thm:numoforbits}, we give the following
lemma. 
\begin{lem}\label{lem:freeaction}
For $(d_1,d_2,\dots,d_q)$ with $d_i \ge 3$ for all $i=1,\ldots,q$, 
put $N':=\lcm(d_1,\ldots,d_q)$,
$G:=(\Z/N'\Z)^\times$ and 
\[
H:=\{ g \in G \mid \exists(\varepsilon_i) \in \{ \pm 1 \}^q
\text{ such that } g\equiv \varepsilon_i \mod d_i (i=1,\ldots,q) \}.
\]
Then, 
the quotient group $G/H$ acts on $J_{d_1} \times \cdots \times J_{d_q}$
freely. In particular, every $G/H$-orbit on $J_{d_1} \times \cdots \times J_{d_q}$
has $|G/H|$ elements.  
The cardinality of $H$ is $2^{\beta_0(\Gamma(V))}$, where 
 $V=\{1,2,\dots,q\}$ and $\beta_0(\Gamma(V))$ is the $0$th
 Betti number, i.e., the number of connected components of $\Gamma(V)$. 
 \end{lem}
 \begin{proof}
We express the decomposition into connected components as 
$\{1,\ldots, q\} = \sqcup_{k=1}^{\beta_0} A_k$, where
  $\beta_0$ represents $\beta_0(\Gamma(V))$.
Note that 
\[
H = \{ g \in G \mid \exists(\varepsilon_k) \in \{\pm 1\}^{\beta_0}
\text{ such that } g \equiv \varepsilon_k \mod d_i \mbox{ for } i \in A_k \}.
\]
In fact, we see that $\varepsilon_i \equiv g \equiv \varepsilon_j \mod \gcd(d_i,d_j)$.
Thus, if $\gcd(d_i,d_j) \ge 3$, then $\varepsilon_i = \varepsilon_j$. 
This implies the above identity.
Put $N_k :=\lcm(d_i \mid i \in A_k)$ for $k=1,\ldots, \beta_0$.
Then, we can write 
\begin{align*}
H &= \{ g \in G \mid \exists(\varepsilon_k) \in \{\pm 1\}^{\beta_0}
\mbox{ such that } g \equiv \varepsilon_k \mod N_k \}.  
\end{align*}
Also, it is seen that $\gcd(N_k,N_{k'})\le 2$ for $k\neq
  k'$. 
It follows that the map 
\[
H \rightarrow (\varepsilon_k) \in \{\pm 1\}^{\beta_0}
\]
is bijective.
\end{proof}

\begin{proof}[Proof of \tref{thm:numoforbits}]
For simplicity, 
$\vec{d}$ is assumed to be rearranged 
in such a way that 
$d_1 \ge \dots \ge d_{q'} \ge 3 > d_{q'+1} \ge \dots \ge
 d_q$, and we write $\vec{d'} = (d_1,\dots,d_{q'})$.  
It is clear that $\orb(\vec{d}) = \orb(\vec{d'})$ since
 $J_1$ and $J_2$ are singletons, 
and 
since $\tilde{\phi}(1) = \tilde{\phi}(2) = 1$ and 
$\tilde{\phi}(\lcm(m,2)) = \tilde{\phi}(m)$ for any $m \in \N$, we have 
\[
\frac{\prod_{i=1}^{q} 
 \tilde{\phi}(d_i)}{\tilde{\phi}(\lcm(d_1,d_2,\dots,d_q))} 
= \frac{\prod_{i=1}^{q'} 
 \tilde{\phi}(d_i)}{\tilde{\phi}(\lcm(d_1,d_2,\dots,d_q'))}. 
\] 
Therefore, it suffices to consider the case $d_i \ge 3$ for all $i =1,2,\dots,q$. 
From \lref{lem:freeaction}, it follows that 
\[
 \orb(\vec{d}) = \frac{|J_{d_1} \times \cdots \times J_{d_q}|}{|G/H|} 
= \frac{\prod_{i=1}^{q} 
 \tilde{\phi}(d_i)}{\tilde{\phi}(\lcm(d_1,d_2,\dots,d_q))} 2^{\beta_0(\Gamma(V))-1}.  
\]
This completes the proof. 
\end{proof}

We next 
give a criterion for the irreducibility of $\Psi_{\vec{d}}(x)$, which is later 
used to obtain a condition for $\Psi_{\vec{d}}(x ; \cO)$
defined in \eqref{eq:Psiorbit} to factor into the powers of a linear function for $q=2$. 

\begin{prop}\label{proposition:injective}
\textup{(1)} If $\mathcal{O} \subset J_{d_1} \times \cdots \times J_{d_q}$
is stable under the action of $(\Z/N\Z)^\times$, 
then $\Psi_{\vec{d}}(x; \mathcal{O}) \in \Z[x]$. \\
\textup{(2)} Let $\mathcal{O} \subset J_{d_1} \times \cdots \times J_{d_q}$
be a $(\Z/N\Z)^\times$-orbit. 
Then the map 
\begin{equation}\label{equation:injective}
c: \mathcal{O} \ni (j_1,\ldots,j_q) \mapsto 2 \sum_{i=1}^q \cos \frac{2\pi j_i}{d_i} \in \R
\end{equation}
is injective if and only if $\Psi_{\vec{d}}(x; \mathcal{O}) \in \Z[x]$ is irreducible.  \\
\end{prop}
\begin{proof}
The Galois group $\Gal(\Q(\zeta_N)/\Q)$ is identified with
$(\Z/N\Z)^\times$, where $\zeta_N=\exp(2\pi\sqrt{-1}/N)$ is
an $N$th primitive root of unity.
The Galois action of an element $a \in  (\Z/N\Z)^\times$ to
an element $c(j_1,\ldots,j_q) \in \Q(\zeta_N)$ is given by
the action \eqref{equation:GaloisAction}.
If $\mathcal{O}$
is stable under the action of $(\Z/N\Z)^\times$, 
then $\Psi_{\vec{d}}(x; \mathcal{O}) \in \Q[x]$.
Since the coefficients of $\Psi_{\vec{d}}(x; \mathcal{O})$ are algebraic integers,
we obtain the first assertion (1).
If the map $c$ is not injective, then $\Psi_{\vec{d}}(x;
 \mathcal{O})$ has a multiple root,
and hence it cannot be irreducible.
Conversely, if the map $c$ is injective on an orbit $\mathcal{O}$,
then the action of $(\Z/N\Z)^\times$ on the roots of $\Psi_{\vec{d}}(x; \mathcal{O})$
is transitive, and therefore the polynomial $\Psi_{\vec{d}}(x; \mathcal{O})$ is irreducible.
\end{proof}

We now give an example of Proposition~\ref{proposition:injective}(2).
\begin{cor}\label{lem:inj2irr}
Suppose $d_1,\ldots,d_q$ are relatively prime and $d_i \ge 3$ for all $i=1,2,\dots,q$.
Then the action of $(\Z/N\Z)^\times$ on 
$J_{d_1} \times\cdots\times J_{d_q}$ is transitive,
the map $c$ in \eqref{equation:injective} is injective, 
and $\Psiprime_{\vec{d}}(x) \in \Z[x]$ is irreducible.
\end{cor}
\begin{proof}
For any $(j_1,\ldots,j_q) \in J_{d_1} \times \cdots \times J_{d_q}$, 
there exists $g \in \Z$ such that
$j_i \equiv g \mod d_i$ for all $i=1,\ldots,q$ from the 
Chinese Remainder Theorem.
This shows that $(j_1,\ldots,j_q)=(g 1,\ldots, g 1)$ belongs to an orbit of $(1,\ldots,1)$,
and thus the transitivity follows. 
The injectivity is proved as follows. 
First, suppose $c(j_1,\ldots,j_q) = c(j'_1,\ldots,j'_q)$.
Then,  
\begin{align*}\label{equation:rational_identity}
-\cos(2\pi j_1/d_1) + \cos(2\pi j'_1/d_1)
& = \sum_{i=2}^q \left( \cos(2\pi j_i/d_i) - \cos(2 \pi j'_i/d_i) \right)
\end{align*}
is an element of 
$\Q(\zeta_{d_1}) \cap \Q(\zeta_{d_2\cdots d_q}) = \Q$
since $d_1$ and $d_2\cdots d_q$ are coprime. 
There exists a rational number $b$ such that
$-\cos(2\pi j_1/d_1) + \cos(2\pi j'_1/d_1)=b$.
Taking the Galois conjugates, and summing up over $J_{d_1}$,
 we have
\[
\tilde{\phi}(d_1) b
=
-\sum_{j_1 \in J_{d_1}} \cos \frac{2\pi j_1}{d_1} + \sum_{j'_1 \in J_{d_1}} \cos \frac{2\pi j'_1}{d_1} 
=0.
\]
This implies $b=0$ and $j'_1=j_1$. 
Then, by induction on $q$, we have the injectivity. 
The irreducibility follows from \pref{proposition:injective}(2).
\end{proof}

The polynomial $\Psi_{\vec{d}}(x; \mathcal{O})$ may not be 
irreducible. The following lemma clarifies the situation in this regard. 
\begin{lem}\label{lem:fiber}
 Let $\mathcal{O} \subset J_{d_1} \times \cdots \times J_{d_q}$ be a $(\Z/N\Z)^\times$-orbit. 
Then the fibers of the map $c$ in \eqref{equation:injective}
 have the same cardinality.  
In particular, there exist an irreducible polynomial
$\Psiirred_{\vec{d}}(x; \mathcal{O}) \in \Z[x]$
and a number $m_{\mathcal{O}} \in \N$
such that
$\Psi_{\vec{d}}(x; \mathcal{O})
= \Psiirred_{\vec{d}}(x; \mathcal{O})^{m_{\mathcal{O}}}$.
\end{lem}
\begin{proof}
Take $(j_{10},\ldots,j_{q0})\in \mathcal{O}$ so that
the cardinality of the fiber of the map (\ref{equation:injective}) is maximum.
Denote $y_0 := c(j_{10},\ldots,j_{q0}) \in \R$.
For any $y \in c(\mathcal{O})$,
there exists a $g \in (\Z/N\Z)^\times$ such that
$y=c(gj_{10},\ldots,gj_{q0})$ since $\mathcal{O}$ is an orbit.
We see that the map
\[
c^{-1}(y_0) \ni (j_1,\ldots,j_q) \mapsto (gj_1,\ldots,gj_q) \in c^{-1}(y)
\]
is a well-defined injective map.
By the choice of maximality,
this map turns out to be bijective.
This proves the first assertion.
Letting $m_{\mathcal{O}}$ denote the common
 cardinality of the fiber and putting 
\[
\Psiirred_{\vec{d}}(x; \mathcal{O}) = \prod_{y \in c(\mathcal{O})} (x- y),
\]
we have the second assertion.
\end{proof}

\begin{ex}
In the case where $q=2$ and $\vec{d}=(5,5)$, we have two
 orbits. One is `diagonal' and the other is `off-diagonal':  
\begin{align*}
\mathcal{O}_1 = \{ (1,1), (2,2) \}, \quad
\mathcal{O}_2 = \{ (1,2), (2,1) \}.
\end{align*}
Note that $c$ in \eqref{equation:injective} maps the orbits
as follows: 
\[
\mathcal{O}_1 \mapsto \{4 \cos \frac{2 \pi}{5}, 4 \cos \frac{4 \pi}{5}\}, \quad 
\mathcal{O}_2 \mapsto \{-1,-1\}. 
\]
The latter is not injective, while the former is. 
We see that $\Psi_{\vec{d}}(x; \mathcal{O}_1)=x^2+2x-4$ is irreducible, 
and $\Psi_{\vec{d}}(x; \mathcal{O}_2)=(x+1)^2$ is reducible, 
 that is, $\Psiirred_{\vec{d}}(x; \mathcal{O}_2)=x+1$,
$m_{\mathcal{O}_1}=1$ and $m_{\mathcal{O}_2}=2$. 
\end{ex}

In what follows, we focus on the case $q=2$, 
in which we know more about the reducibility of $\Psi_{\vec{d}}(x ; \cO)$. 
\begin{prop}\label{prop:linear}
Suppose $q=2$.
If the degree of $\Psiirred_{d_1, d_2}(x; \mathcal{O})$ is one,
the possibilities for $d_1, d_2$ and the orbits $\cO$ are as
 follows: 
\[
\begin{array}{|c|c|c|c|}
\hline
(d_1,d_2) & \cO & \Psiirred_{d_1,d_2}(x; \cO) & \mbox{condition} \\
\hline \hline
(d_1,d_2) & \cO_{1,1} & x-\lambda & d_1,d_2=1,2,3,4,6 \\
(m,m) & \cO_{1,m/2-1} & x & 4|m \\
(m,2m) & \cO_{1,m-2} & x & m:\mbox{odd}  \\
(5,5) & \cO_{1,2} & x+1 & \\
(10,10) & \cO_{1,3} & x-1 & \\
\hline
\end{array}
\]
In this table, $\lambda = 2\cos \frac{2\pi}{d_1} + 2 \cos
 \frac{2\pi}{d_2} \in \Z$ and 
$\cO_{1,a}$ is the orbit containing $(1,a) \in J_{d_1}
 \times J_{d_2}$. More explicitly, 
\begin{align*}
& \cO_{1,1} = J_{d_1} \times J_{d_2}, \quad  
\cO_{1,m/2-1} = \{ (j, m/2-j) \mid j \in J_m\}, \\  
& \cO_{1,m-2} = \{ (j, m-2j) \mid j \in J_m \}, \\  
& \cO_{1,2}, \cO_{1,3}= \{ (j_1,j_2) \in J_{d_1} \times
      J_{d_2} \mid j_1 \neq j_2 \}. 
\end{align*}
\end{prop}

By \lref{lem:fiber}, if the degree of $\Psiirred_{d_1, d_2}(x; \mathcal{O})$
is one, then the image of the map $c : \cO \to \Q$ is a
singleton. We now derive some necessary conditions for this 
to be the case. 
\begin{lem}\label{lem:necessary} 
Let $g=\gcd(d_1,d_2)$ 
and define $g_i$ to be the product of all factors of $d_i$
 in common with $g$. Put $m_i := d_i / g_i \in \Z$. 
If the degree of $\Psiirred_{d_1, d_2}(x; \mathcal{O})$
is one, or equivalently, if the image of the map $c : \cO \to
 \Q$ is a singleton, then the following conditions hold:  
\begin{enumerate}
 \item[(i)] $g_i \le 2$ or $m_i \le 2$ for $i=1,2$. 
 \item[(ii)] $m_1, m_2 \in \{1,2,3,4,6\}$. 
 \item[(iii)] $g_1=g_2$ or $(g_1,g_2) = (2,4)$ or $(4,2)$. 
\end{enumerate}
\end{lem}
\begin{proof}
Let $N'=\lcm(d_1,d_2)$. 
Suppose there exist $d_1,d_2 \in \N$ and $(j_1,j_2) \in J_{d_1} 
 \times J_{d_2}$ such that 
\[
2\cos \frac{2\pi j_1 b}{d_1} + 2 \cos \frac{2\pi j_2 b}{d_2} =
 \lambda\in \Q \quad \text{for $\forall b \in \Z/N'\Z$}. 
\]
We claim that every $b \in \Z$ with $b \equiv 1 \mod d_1$
and $\gcd(b,d_2)=1$ satisfies $b \equiv \pm 1 \mod d_2$. 
Indeed, 
since $\cos \frac{2\pi j_1 b}{d_1} = \cos \frac{2\pi j_1}{d_1}$,
we have $\cos \frac{2\pi j_2 b}{d_2}= \cos \frac{2\pi j_2
 }{d_2}$, which implies that $b \equiv \pm 1 \mod d_2$. 

Now we use the claim above for three cases regarding the
 value of $b$. 
Note that $d_i = m_i g_i$, and if $\tilde{g}=\lcm(g_1,g_2)$, then
$\tilde{g}, m_1$ and $m_2$ are mutually prime. 

First, consider $b \in \Z$ with
$b \equiv 1 \mod m_1 \tilde{g}$
and $b \equiv -1 \mod m_2$.
By the claim, we have $b \equiv \pm 1 \mod d_2 = g_2 m_2$.
If $b \equiv 1$, then we have $m_2 \le 2$, 
while if $b \equiv -1$, then we have $g_2 \le 2$. 
We conclude that either $m_2 \le 2$ or $g_2 \le 2$.

Second, consider $b \in \Z$ with $b \equiv 1 \mod m_1 \tilde{g}$ and $\gcd(b,m_2) =1$. 
By the claim, we have $b \equiv \pm1 \mod m_2$. 
This shows that $\varphi(m_2) \le 2$, 
and we conclude that $m_2 = 1,2,3,4$ or $6$.

Third, consider $b \in \Z$ with
$b\equiv 1 \mod m_1 m_2$
and $b\equiv 1+g_1 \mod g_2$.
By the claim, we have $b \equiv \pm 1 \mod g_2$.
If $b \equiv 1$, then $g_2 | g_1$, while 
if $b \equiv -1$, then $g_2 | (g_1+2)$.
Since the conditions on $d_1$ and $d_2$ are symmetric,
we also have $g_1 | g_2$ or $g_1 | (g_2+2)$.
We conclude that either $g_1= g_2$ or 
$\{ g_1, g_2 \} = \{ 2,4\}$.
\end{proof}

Now we use \lref{lem:necessary} (i), (ii) and (iii) to prove \pref{prop:linear}. 

\begin{proof}[Proof of \pref{prop:linear}]

First, we note that $\Psi_{d_1,d_2}(x) =
\Psi_{d_2}(x-\alpha)$ with $\alpha = 2\cos \frac{2\pi}{d_1}
\in \Z$ for $d_1 \in \{1,2,3,4,6\}$ from the remark in 
\eref{ex:table}. 
Since $\Psi_d(x)$ is irreducible for any $d \in \N$ from
 \lref{lem:Psid}, 
so is $\Psi_{\vec{d}}(x)$ in this case. 
The degree of $\Psiirred(x ; \cO)$ is one if and only if 
that of $\Psi_{d_2}(x)$ is one, which is true if and only if
 $d_2 \in \{1,2,3,4,6\}$. 
This gives the first line in the table. 
Indeed, in all of the cases above, we have only one orbit
 $\cO_{1,1}$ on $J_{d_1} \times J_{d_2}$. 

In what follows, we can assume that $d_1, d_2 \not\in
 \{1,2,3,4,6\}$, and hence $\tilde{\phi}(d_1), 
 \tilde{\phi}(d_2) \ge 2$.  

In the case $g=1$ with $g_1=g_2=g$ in (iii), we have $m_i = d_i \in
 \{1,2,3,4,6\}$ by (ii). 
In the case $g_i = 2$ in (iii) for $i=1$ or $2$, $m_i$ should be odd, and hence 
choices are only $m_i =1,3$ by (ii). 
This implies $d_i = 2$ or $6$. 
Both of these cases have already been excluded. 

If $g\ge 3$ with $g_1=g_2=g$, then $m_i \le 2$.
The only choices in this case are $(m_1,m_2) =(1,1), (1,2), (2,1)$.
Furthermore, if $m_i=2$, then $g$ must be odd.
It follows that $(d_1,d_2) = (g,2g)$ with odd $g\ge 5$ or
 $(d_1,d_2)=(g,g)$ for arbitrary $g\ge 5$ (with $g=6$ excluded).

To this point, we have not used the orbit structure. 
Now we consider the choice of $(j_1, j_2) \in \cO$.
Taking the sum of all the terms under the Galois action of
 $(\Z/d_2 \Z)^\times$, and using the same argument as in the
 proof of  \cref{lem:inj2irr},  
we have
\[
\tilde\varphi(d_2) \lambda 
= 2 \sum_{j_1 \in J_{d_1}} \cos \frac{2\pi j_1}{d_1} + 2 \sum_{j_2 \in J_{d_2}} \cos \frac{2\pi j_2}{d_2}
= \mu(d_1) + \mu(d_2)
\]
for $d_1, d_2 \ge 3$, where $\mu$ is the M\"obius function.
Since $\lambda$ is an algebraic integer,
we see that $\Z \ni \lambda = \frac{\mu(d_1) + \mu(d_2)}{\tilde\varphi(d_2)}$.
This shows either $\lambda=0$ or $\tilde\varphi(d_2) \le 2$.

If $\la \not=0$, then $\tilde{\phi}(d_2) =2$, i.e., 
$d_2 \in \{5,8,10,12\}$. In this case, we have  
$\lambda=\mu(d_1)=\mu(d_2)=\pm 1$, and thus $d_2$ is not a multiple of $4$
and $(d_1,d_2) \neq (g,2g)$.
From this we obtain $d_1=d_2 \in \{5, 10\}$ and the last two lines in the table. 

Now we consider the case $\lambda=0$.
The condition $2\cos \frac{2\pi j_1}{d_1} + 2 \cos \frac{2\pi j_2}{d_2}=0$
implies $j_1/d_1 + j_2/d_2 = 1/2$ since $d_1,d_2 \ge 3$. 
If $d_1=d_2$, then $j_1+j_2=d_1/2$.
It follows that $d_1$ is even, and hence that $j_1$ and $j_2$ are odd.
This implies that $d_1/2$ is even, and therefore $4|d_1$.
For $j_1=1$, we have $j_2=d_1/2-1 \in J_{d_2}$.
This gives the second line in the table.
If $2d_1=d_2$, then $2j_1+j_2=d_1$.
For $j_1=1$, we have $j_2=d_1-2 \in J_{d_2}$.
This gives the third line in the table. 
\end{proof}

Now we give three examples for $q=2$: (I) $(d_1,d_2)
= (m,2m)$; (II) $(d_1, d_2)=(m,m)$; (III)
$\tilde\varphi(d_2)=2$. 
From these examples, we obtain three \textit{observations} 
with the aid of numerical computations.

Before studying each case, we give an elementary remark
on the representatives for an orbit decomposition. 
\begin{lem}\label{lem:representative_product}
Suppose that $G$ acts on $X$ simply transitively.
Fix an arbitrary $x_0 \in X$. 
Then, for any $Y$ (with $G$ action),
the set $\{ x_0 \} \times Y$ consists of complete representatives of
the action of $G$ on $X \times Y$.
\end{lem}
\begin{proof}
For any $(x, y) \in X \times Y$,
there exists a $g \in G$ such that $gx=x_0$.
Then $g(x,y) = (x_0, gy)$.
If $(x_0,y)$ and $(x_0,y')$ belong to the same
$G$-orbit, then $y=y'$, because
$g(x_0,y)=(x_0,y')$ implies $g=e$, 
the identity element of the group $G$.
\end{proof}

\noindent
(I) For $q=2$, $(d_1,d_2)=(m,2m)$. 

We start with the orbit decomposition.
\begin{lem}
The number of orbits is $\orb((m,2m)) = \tilde{\phi}(m)$, 
 and each orbit has $\tilde{\phi}(2m)$ elements. 
The orbit decomposition is given by 
\[
 J_m \times J_{2m} 
= \begin{cases}
\sqcup_{a \in J_m} \cO_{1,a} & \text{for even $m$}, \\
\sqcup_{a \in J_{2m}} \cO_{1,a} & \text{for odd $m$}. 
  \end{cases}
\]
\begin{proof}
The first part follows from \eref{ex:numberoforbits} (i) and 
the second part follows from
 \lref{lem:representative_product}.
If $m \ge3$ is odd, 
then $\tilde{\phi}(2m) = \tilde{\phi}(m)$
and the natural map induces the group isomorphism
$(\Z/2m\Z)^\times / \{ \pm 1\}
\overset{\sim}{\rightarrow} 
(\Z/m\Z)^\times / \{ \pm 1\}$,
so the action of $(\Z/2m\Z)^\times / \{ \pm 1\}$
on $J_m$ is simply transitive. 
If $m$ is even, it is easlily seen that 
$(1,i) \in \cO$ if and only if $(1,m-i) \in \cO$ for $i \in
 J_{2m}$, from which we have the decomposition. 
\end{proof}
\end{lem}

\begin{obs}
(i) $\Psi_{m,2m}(x ; \cO)$ seems to be irreducible
except for odd $m$ with $\cO=\cO_{1,m-2}$ appearing in \pref{prop:linear}. \\
(ii) If $\Psi_{m,2m}(x ; \cO) = \Psi_{m,2m}(x; \cO')$,
then $\cO=\cO'$.
\end{obs}

\noindent
(II) For $q=2$, $(d_1,d_2)=(m,m)$. 

We define the involution $\iota$ on $J_m \times J_m$ by
$\iota(x,y) = (y,x)$. It is obvious from this definition that 
$\Psi_{\mm}(x; \iota(\cO))= \Psi_{\mm}(x; \cO)$.
\begin{lem}\label{lem:mmcase} 
(i) The number of orbits is $\orb((m,m)) = \tilde{\phi}(m)$
 and each orbit has $\tilde{\phi}(m)$ elements.
The orbit decomposition is given by 
\[
 J_m \times J_m 
= \bigsqcup_{a \in J_m} \cO_{1,a}.  
\]
(ii) The diagonal orbit $\cO_{1,1} = \{(i,i) : i \in J_m\}$ is invariant under
 $\iota$. If $\tilde\varphi(m)$ is odd, then there is no non-diagonal $\iota$-invariant orbit.
\end{lem}
\begin{proof}
(i) The first part follows from \eref{ex:numberoforbits}. 
The second part follows from
 \lref{lem:representative_product} by setting $x_0=1$ and $Y = J_m$. \\
(ii) Any non-diagonal $\iota$-invariant orbit
 has no fixed point under the action of $\iota$. 
In particular, the cardinality of such an orbit is even. 
The second assertions follow. 
\end{proof}
If $\tilde\varphi(m)$ is even, 
then the situation is different. 
Note that if $4|m$ with $m>4$, then $\tilde\varphi(m)$ is even. 
\begin{lem}
(i) Suppose $4|m$.
The orbit $\cO_{1,m/2-1}$ is $\iota$-invariant and 
$\Psi_{\mm}(x ;\cO_{1,m/2-1}) = x^{\tilde\varphi(m)}$. \\
(ii) Suppose $\tilde{\phi}(m)$ is even.
For a non-diagonal $\iota$-invariant orbit $\cO$ with $\cO \neq \cO_{1,m/2-1}$, there exists
       $\Psi_{\mm}^{\half}(x;\cO) 
       \in \Z[x]$ of degree $\tilde{\phi}(m)/2$  such that 
       $\Psi_{\mm}(x ;\cO) = \Psi_{\mm}^{\half}(x;\cO)^2$.
\end{lem}
\begin{proof}
(i) This case is considered in \pref{prop:linear}. \\
(ii) This can be shown in the same way as the proof of
 \lref{lem:mmcase} (ii). 
\end{proof}

We denote by $A$ the number of  $\iota$-invariant orbits.
As we have seen, $A=1$ if $\tilde\varphi(m)$ is odd. 
We have the following formula for $A=A(m)$ if
$\tilde\varphi(m)$ is even:
\begin{lem} 
Suppose $\tilde\varphi(m)$ is even. 
Put
\[
f_1=f_1(m) := \begin{cases}
0 & \mbox{if $4 \nmid m$}, \\
1 & \mbox{if $4\mid m$ and $8 \nmid m$}, \\
2 & \mbox{if $8\mid m$}.
\end{cases}
\]
Let $f_2$ be the number of odd prime factors of $m$. 
Let $f_3 = f_3(m) = 1$ if $4 \nmid m$, $f_2(m) \ge 1$, and $p\equiv 1 \mod 4$ for
 every odd prime factor $p$ of $m$; otherwise, we put $f_3(m)=0$.
Then we have 
\[
A=2^{f_1+f_2+f_3-1}.
\]
\end{lem}
\begin{proof}
We consider the orbits that are invariant under the involution $\iota$.
We will solve the equation $x^2\equiv \pm1 \mod m$.
By the Chinese Remainder Theorem,
this equation can be reduced to the equation for each prime factor.
If $m=2^e$, then the number of solutions
to the equation $x^2\equiv 1 \mod m$ is given by $2^{f_1(m)}$.
If $m=p^2$, then we have two solutions to 
$x^2 \equiv 1 \mod m$.
This contributes to $f_2$.
We have a solution to the equation
$x^2 \equiv -1 \mod m$ only in the case $f_3(m)=1$. 
This solution makes double the number of 
solution for the equation $x^2 \equiv \pm 1 \mod m$ 
 to the number of solutions to the equation $x^2 \equiv 1
 \mod m$. 
Finally, we parameterize the orbits as $\cO_{1,x}$ with $x \in J_m$,
so that we will divide by 2. 
The formula for $A(m)$ given above follows. 
\end{proof}

We now proceed from orbits to polynomials.

\begin{obs}
(i) For an orbit $\cO \neq \iota(\cO)$, 
 $\Psi_{\mm}(x ;\cO)$  seem to be irreducible. 
In particular, $\Psi_{\mm}(x ;\cO)$ seem to be irreducible
if $\tilde{\phi}(m)$ is odd. \\
(ii) $\Psi_{\mm}^{\half}(x;\cO)$ seems to be irreducible
 for every $\iota$-invariant orbit $\cO\neq\cO_{1,m/2-1}$. 
This happens only if $\tilde{\phi}(m)$ is even. \\
(iii) If $\Psi_{\mm}(x ; \cO)=\Psi_{\mm}(x ; \cO')$,
then $\cO'=\cO$ or $\cO'=\iota(\cO)$. \\
\end{obs}

\noindent
(III) For $q=2$, $\tilde{\phi}(d_2)=2$. 

In this case, 
$\tilde{\phi}(d_2) = 2$ implies $d_2=5, 8, 10$ or $12$, for which 
we set $a = 2, 3, 3$ and $5$, respectively. Then $J_{d_2} = \{1,a\}$. 

\begin{lem} Suppose $d_2 = 5, 8, 10$ or $12$. 
 If $5 | d_1$ for $d_2 =10$ or 
$d_2 | d_1$ for $d_2 \not=10$, then we have 
\[
 J_{d_1} \times J_{d_2} = \cO_{1,1} \sqcup \cO_{1,a}. 
\]
Otherwise, $J_{d_1} \times J_{d_2}$ forms a single orbit.
\end{lem}
\begin{proof}
In these cases, we know that 
$\orb(d_1,d_2) = \tilde{\phi}(d_2) = 2$ from
 \eref{ex:numberoforbits}(i). 
Therefore, the orbit decomposition must be the one given
 above. 
\end{proof}

\begin{obs}
If $\tilde{\phi}(d_2) = 2$, 
then $\Psi_{d_1,d_2}(x; \cO)$ seems to be irreducible except
 in the cases considered in \pref{prop:linear}. 
\end{obs}

\vspace{1cm}
\noindent
{\bf Acknowledgment}\\
\indent
This work was partially supported by JST CREST Mathematics 15656429.
H.O. is supported in part by JSPS Grants-in-Aid 15H03613 and 17K18726, 
and T.S. is supported in part by JSPS Grants-in-Aid 17K18740 and 18H01124.


\vskip 1cm

\end{document}